\documentclass{amsart}
\usepackage[margin=1.5in]{geometry}
\usepackage{amsthm}
\usepackage{amssymb}
\usepackage[non-compressed-cites]{amsrefs}
\usepackage{esint}
\usepackage{caption}
\usepackage{enumitem}
\usepackage{mathtools}
\usepackage{xcolor}
\usepackage{microtype}
\usepackage{multirow}
\usepackage{array}
\usepackage{tikz}
\usepackage{threeparttable}
\usepackage{hyperref}

\usetikzlibrary{arrows.meta,decorations.markings,calc}

\allowdisplaybreaks

\newcommand{\R}{\mathbb{R}}
\newcommand{\Sph}{\mathbb{S}}
\newcommand{\del}{\partial}
\newcommand{\Ric}{\operatorname{Ric}}
\renewcommand{\div}{\operatorname{div}}
\newcommand{\Rm}{\operatorname{Rm}}
\newcommand{\tr}{\operatorname{tr}}

\newcommand{\dvol}{\mathrm{dvol}}
\newcommand\inner[2]{\langle #1, #2 \rangle}
\newcommand{\norm}[1]{\left\lVert #1 \right\rVert}
\renewcommand{\O}{\mathcal{O}}
\renewcommand{\L}{\mathcal{L}}
\newcommand{\Z}{\mathbb{Z}}

\newcommand{\CP}{\mathbb{CP}}
\newcommand{\RP}{\mathbb{RP}}

\newcommand{\I}{\mathcal{I}}
\newcommand{\C}{\mathbb{C}}
\renewcommand{\epsilon}{\varepsilon}

\newtheorem{theorem}{Theorem}[section]
\newtheorem{corollary}[theorem]{Corollary}
\newtheorem{lemma}[theorem]{Lemma}
\newtheorem{definition}[theorem]{Definition}
\newtheorem{proposition}[theorem]{Proposition}

\numberwithin{equation}{section}
 
\hypersetup{
    colorlinks,
    linkcolor={red!60!black},
    citecolor={blue!80!black},
}

\setlist[itemize]{leftmargin=3em}
\setlist[enumerate]{leftmargin=3em}
\mathtoolsset{showonlyrefs}

\title[An analytical characterization of Eguchi--Hanson space]{An analytical characterization of Eguchi--Hanson space and its higher-dimensional analogs}
\author[Michael B. Law]{Michael B. Law}
\address{MIT, Department of Mathematics, 77 Massachusetts Avenue, Cambridge, MA 02139, USA.}
\email{mikelaw@mit.edu}

\begin{document}

\begin{abstract}
    Let $(M,g)$ be a complete 4-dimensional Ricci-flat ALE orbifold with finitely many orbifold points and group at infinity $\Z_2$. We prove that if the $L^2$ kernel of its Lichnerowicz Laplacian has dimension at most 3, then $(M,g)$ is either the Eguchi--Hanson space or the flat orbifold $\R^4/\Z_2$. A similar uniqueness result is proved for Calabi's higher-dimensional analogs of the Eguchi--Hanson space among Ricci-flat K\"ahler ALE orbifolds with group at infinity $\Z_m$.
\end{abstract}

\maketitle

\tableofcontents

\section{Introduction}

Ricci-flat asymptotically locally Euclidean (ALE) spaces lie at the intersection of differential geometry and mathematical physics. In geometry, they arise in the compactness theory and desingularization of Einstein metrics \cites{anderson,bkn89,biquard,morteza-viaclovsky,donaldson-exposition,ozuch}, as well as in the construction of manifolds with special holonomy \cites{joyce,joyce-karigiannis}. In physics, complete 4-dimensional Ricci-flat ALE spaces play a central role in Euclidean quantum gravity \cites{hawking-GI,eguchi-hanson} and form a class of gravitational instantons. They also arise in various string theories \cites{BPS,ALE-string}. Let us first state our definition of an ALE space, which allows for orbifold singularities in a compact set.

\begin{definition} \label{def:ALE}
    An \emph{asymptotically locally Euclidean (ALE) space} is a connected Riemannian orbifold $(M^n,g)$ such that there is a finite subgroup $\Gamma \subset SO(n)$ acting freely on $\Sph^{n-1}$, a compact set $K \subset M$, a radius $R>0$, a diffeomorphism $\Phi: M \setminus K \to (\R^n \setminus B_R(0))/\Gamma$, and a decay rate $\tau > 0$ such that
    \begin{align} \label{eq:ALE-diffeo}
        |\nabla^k(\Phi^*\bar{g} - g)|_g \leq \O(r^{-\tau-k})
    \end{align}
    as $r \to \infty$ for each $k \in \mathbb{N}_0$. Here, $\bar{g}$ is the Euclidean metric on $(\R^n \setminus B_R(0))/\Gamma$, $\nabla$ is the Levi-Civita connection with respect to $g$, and $r$ is the Riemannian distance from a point with respect to $g$. We call $\tau$ the \emph{order} of $(M,g)$, and $\Gamma$ its \emph{group at infinity}. If $\Gamma = \{1\}$, then $(M,g)$ is called \emph{asymptotically Euclidean}.
\end{definition}

The prototypical example of a nontrivial (i.e. nonflat) 4-dimensional Ricci-flat ALE space is \emph{Eguchi--Hanson space} \cite{eguchi-hanson}. This is a complete Ricci-flat ALE metric on the cotangent bundle of $\Sph^2$ with group at infinity $\Z_2$. The metric is hyperk\"ahler and $U(2)$-invariant with cohomogeneity one. See \S\ref{subsec:EH} for further details. We remark that in addition to its underlying physical motivations, Eguchi--Hanson space appears as a common building block in geometric gluing constructions; see for instance \cites{biquard,brendle-kapouleas,donaldson-exposition,joyce-karigiannis}.

Soon after Eguchi--Hanson space was discovered, Gibbons and Hawking \cite{gibbons-hawking} produced a large class of hyperk\"ahler (thus Ricci-flat) gravitational instantons, including ALE spaces with groups at infinity $\Z_k$ ($k \geq 1$). The latter spaces were also constructed by Hitchin using twistor methods \cite{hitchin-gravitons}. Kronheimer \cite{kronheimer} then unified the known constructions of ALE spaces: he constructed and classified all simply-connected hyperk\"ahler ALE 4-manifolds as hyperk\"ahler quotients, exhibiting an ADE correspondence and identifying each space with the minimal resolution of $\C^2/\Gamma$ for $\Gamma \subset SU(2)$.

Motivated by these developments, Bando, Kasue, and Nakajima \cite{bkn89} conjectured that every complete simply-connected Ricci-flat ALE 4-manifold $(M,g)$ is hyperk\"ahler and hence falls under Kronheimer's classification. While this conjecture remains open in general, it is known to hold under additional assumptions. Nakajima \cite{nakajima-spin} proved it in the affirmative assuming $M$ is spin and $\Gamma \subset SU(2)$. Lock and Viaclovsky \cite{lock-viaclovsky} proved that if $M$ is diffeomorphic to a minimal resolution of $\C^2/\Gamma$ or an iterated blowup thereof, where $\Gamma \subset U(2)$, then $\Gamma \subset SU(2)$ and $g$ is hyperk\"ahler. More recently, there has been progress on uniqueness of Eguchi--Hanson space, which is a special case of the Bando--Kasue--Nakajima conjecture. Li \cite{li-4dale} proved that if $(M,g)$ is a 4-dimensional Ricci-flat ALE manifold with $\Gamma \subset SU(2)$ and admits a Hermitian non-K\"ahler structure, then $(M,g)$ is Eguchi--Hanson space with the opposite orientation. Araneda and Lucietti \cite{araneda-lucietti} proved a similar uniqueness result assuming $(M,g)$ is toric, Hermitian and non-K\"ahler, without any assumptions on $\Gamma \subset SO(4)$.

In this paper, we take a different angle. Rather than imposing global symmetry, we establish uniqueness of certain Ricci-flat ALE spaces under an analytical condition on the Lichnerowicz Laplacian $\Delta_L$. On a Ricci-flat space, $\Delta_L$ operates on symmetric 2-tensors $h$ by
\begin{align} \label{eq:lichnerowicz-laplacian}
    \Delta_L h = \Delta h + 2\Rm(h),
\end{align}
where $(\Delta h)_{kl} = g^{ij} \nabla^2_{ij} h_{kl}$ is the rough Laplacian and $\Rm(h)_{ik} = \Rm_{ijkl}h^{jl}$ is the action of the curvature tensor. The operator $\Delta_L$ is (a multiple of) the linearized Ricci curvature in a suitable gauge, and its kernel elements represent metric variations that preserve the Ricci-flat condition to first order.

On ALE spaces, it is natural to focus on the $L^2$ kernel $\ker_{L^2}(\Delta_L)$. In gluing problems for Einstein metrics, such kernel elements account for potential obstructions to resolving orbifold singularities by gluing in ALE manifolds. In Euclidean quantum gravity, elements of $\ker_{L^2}(\Delta_L)$ are known as \emph{normalizable zero modes} of the gravitational instanton, and they influence the expansion of the gravitational path integral.

Our first result is a classification of 4-dimensional Ricci-flat ALE spaces with group at infinity $\Z_2$ when $\ker_{L^2}(\Delta_L)$ is sufficiently small in dimension.
\begin{theorem} \label{thm:EH-characterization}
    Let $(M^4,g)$ be a complete 4-dimensional Ricci-flat ALE space with finitely many orbifold points and group at infinity $\Z_2$. If $\dim\ker_{L^2}(\Delta_L) \leq 3$, then $(M,g)$ is isometric to either the Eguchi--Hanson space or the flat orbifold $\R^4/\Z_2$.
\end{theorem}
The flat $\R^4/\Z_2$ has $\dim\ker_{L^2}(\Delta_L) = 0$, while $\dim\ker_{L^2}(\Delta_L)=3$ is achieved on Eguchi--Hanson space (see \S\ref{subsec:EH}). By Theorem \ref{thm:EH-characterization}, no intermediate examples exist. Thus, any counterexample to the Bando--Kasue--Nakajima conjecture with group at infinity $\Z_2$ must have $\dim\ker_{L^2}(\Delta_L) \geq 4$, and therefore carries a correspondingly subtler deformation theory.

The regularity assumption in Theorem \ref{thm:EH-characterization} is natural in applications. For instance, it holds for Ricci-flat ALE spaces arising as bubbles from singularities of noncollapsed Gromov--Hausdorff limits of Einstein 4-manifolds \cites{anderson,bkn89}.

Our second result is a similar theorem for Ricci-flat K\"ahler ALE spaces in higher dimensions. These are defined as follows.
\begin{definition} \label{def:kahler-ALE}
    A K\"ahler orbifold $(M^{2m},g,J)$ of complex dimension $m$ is called a \emph{K\"ahler ALE space} if its underlying Riemannian orbifold $(M,g)$ is an ALE space with group at infinity $\Gamma \subset U(m) \subset SO(2m)$.
\end{definition}
Note that we do not stipulate convergence of the complex structure $J$ to the standard one on $\C^m/\Gamma$.

Among the earliest known Ricci-flat K\"ahler ALE spaces are what we call the \emph{Calabi spaces}. The Calabi space of complex dimension $m \geq 2$ lives on the total space of the canonical line bundle $\O_{\CP^{m-1}}(-m)$ of $\CP^{m-1}$, and the group at infinity is $\Z_m \subset SU(m)$ generated by an $m$-th root of unity times the identity. It is equipped with a complete $U(m)$-invariant cohomogeneity one K\"ahler metric. When $m=2$, the Calabi space coincides with Eguchi--Hanson space and is therefore hyperk\"ahler, but when $m \geq 3$ the Calabi space is not hyperk\"ahler. See \S\ref{subsec:calabi} for further details.

Morteza and Viaclovsky \cite{morteza-viaclovsky} computed that $\dim\ker_{L^2}(\Delta_L)=1$ on all Calabi spaces of complex dimension $m \geq 3$. We will prove an analog of Theorem \ref{thm:EH-characterization} for these spaces.

\begin{theorem} \label{thm:calabi-characterization}
    Let $(M^{2m},g,J)$ be a complete Ricci-flat K\"ahler ALE space with complex dimension $m \geq 3$ and group at infinity $\Z_m$. Assume the orbifold singular set has real codimension at least 4. If
    \begin{align} \label{eq:calabi-thm-bound}
        \dim\ker_{L^2}(\Delta_L) \leq \begin{cases}
            2m-3 & \text{if } m \neq 4, \\
            2 & \text{if } m = 4,
        \end{cases}
    \end{align}
    then $(M,g,J)$ is biholomorphically isometric to either the Calabi space or the flat K\"ahler orbifold $\C^m/\Z_m$.
\end{theorem}

Several remarks are in order. Firstly, Theorem \ref{thm:calabi-characterization} can be viewed as a gap theorem: under the stated regularity assumptions, no Ricci-flat K\"ahler ALE space asymptotic to $\C^m/\Z_m$ (where $m=3$ or $m \geq 5$) can have $2 \leq \dim\ker_{L^2}(\Delta_L) \leq 2m-3$, and no Ricci-flat K\"ahler ALE space asymptotic to $\C^4/\Z_4$ can have $\dim\ker_{L^2}(\Delta_L) = 2$.

Secondly, the regularity assumption is consistent with the celebrated regularity theory for noncollapsed Gromov--Hausdorff limits of Einstein metrics \cite{cheeger-naber}.

Thirdly, the proof of Theorem \ref{thm:calabi-characterization} uses the classification of homogeneous metrics on lens spaces $\Sph^{2m-1}/\Z_m$ with large isometry groups. If the isometry group has dimension above some threshold, only Berger metrics can occur. The threshold is relatively higher when $m=4$, so the tighter bound in \eqref{eq:calabi-thm-bound} serves to ensure the threshold is surpassed.

\subsection{Outline of the paper}

In \S\ref{sec:preliminaries}, we lay out notations and conventions, then provide a brief summary of the Eguchi--Hanson and Calabi spaces.

In \S\ref{sec:symmetry-principles}, we work on an arbitrary complete Ricci-flat ALE space with nontrivial group at infinity. We develop a symmetry principle that converts bounds on $\dim\ker_{L^2}(\Delta_L)$ into the existence of global Killing fields tangential to a constant mean curvature foliation of the end. Moreover, the Killing fields integrate to a global isometric Lie group action. This generalizes a symmetry principle developed by the author for Ricci solitons with certain asymptotics \cite{law}, which in turn built on Brendle's seminal work on the Bryant soliton \cites{brendle2013rotational}. Some of the ingredients used later in \S\ref{sec:eh-result}--\ref{sec:calabi-result} are also adapted from \cite{law}, but the case analyses and classification arguments are specific to the ALE setting.

In \S\ref{sec:eh-result}, we prove Theorem \ref{thm:EH-characterization}. Taking $(M,g)$ as in the theorem, we use the results of \S\ref{sec:symmetry-principles} to find a 3-dimensional Lie group $G$ acting isometrically on $(M,g)$. If $G$ acts by cohomogeneity one, then well-known facts about homogeneous metrics on $\Sph^3$ and $\RP^3$ enable us to determine all possibilities for $G$. The structure theory of cohomogeneity one spaces then leads to a diffeomorphism classification of $M$. Aided by results of Lock--Viaclovsky \cite{lock-viaclovsky} and Kronheimer \cite{kronheimer}, this implies $(M,g)$ is either Eguchi--Hanson space or the flat $\R^4/\Z_2$.

On the other hand, if $G$ does not act by cohomogeneity one, then we show that it acts by cohomogeneity two. Lifting to the universal cover $(\tilde{M}_\infty, \tilde{g})$ of the end of $(M,g)$, and analyzing the lifted action, we find that $(\tilde{M}_\infty, \tilde{g})$ is (essentially) isometric to a warped product with 2-dimensional base and round $\Sph^2$ fibers. Using a local variant of Birkhoff's theorem in general relativity, we construct convenient coordinates on an infinite region in the base. From this, we deduce that $\tilde{g}$ is Euclidean Schwarzschild of mass $\mu \in \R$ around an infinite ray. However, the ALE decay rate (Theorem \ref{thm:ALE-decay-rate}) forces $\mu=0$, so $\tilde{g}$ (and hence $g$) is flat on an open set. By unique continuation, $(M,g)$ must be the flat $\R^4/\Z_2$.

In \S\ref{sec:calabi-result}, we prove Theorem \ref{thm:calabi-characterization}. Taking $(M,g,J)$ as in the theorem, we use the results of \S\ref{sec:symmetry-principles} to find a Lie group $G$ of large dimension acting isometrically on $(M,g)$ with cohomogeneity one. Using the classification of homogeneous metrics on certain lens spaces with large isometry groups, we show that the principal orbits of the action are quotient-Berger spheres. The structure theory of cohomogeneity one spaces then leads to a diffeomorphism classification of $M$ and a reduction of $g$ to an ODE system. If $(M,g,J)$ is not the flat $\C^m/\Z_m$, we further use the K\"ahler condition and the form of $g$ to determine $J$ explicitly. This reduces everything to the Calabi ansatz, from which it follows that $(M,g,J)$ is a Calabi space.

\section{Preliminaries} \label{sec:preliminaries}

\subsection{Notation and conventions}

All Riemannian orbifolds are assumed smooth throughout.
Given a Riemannian orbifold $(M,g)$, we will often abbreviate $g = \inner{\cdot}{\cdot}$. Denote by $\nabla$ the Levi-Civita connection, and $\Delta = \tr_g \circ \nabla^2$ the Laplacian on tensor fields. Our convention for the Riemann curvature tensor is
\begin{align} \label{eq:Rm-convention}
    \Rm(X,Y)Z = \nabla_Y \nabla_X Z - \nabla_X \nabla_Y Z + \nabla_{[X,Y]}Z.
\end{align}

By the Myers--Steenrod theorem (see \cite{isom-orb} for orbifolds), the isometry group $\I(M,g)$ of $(M,g)$ is a Lie group acting smoothly on $(M,g)$. Its identity component is denoted by $\I^0(M,g)$. The Lie algebra of $\I(M,g)$, denoted by $\mathfrak{I}(M,g)$, is the space of Killing fields on $(M,g)$, i.e. vector fields $X$ on $M$ such that $\L_X g = 0$ where $\L$ is the Lie derivative.

Given a Lie group $G$ and a closed subgroup $H$ of $G$, the centralizer and normalizer of $H$ in $G$ are denoted $C_G(H)$ and $N_G(H)$ respectively; both are closed subgroups of $G$. We let $C^0_G(H)$ and $N^0_G(H)$ be their identity components. The following facts will be used frequently. Their proofs are rudimentary and are therefore omitted.
\begin{lemma} \label{lem:CN-equality}
    If $G$ is a Lie group and $H$ is a discrete subgroup, then $N_G^0(H) = C_G^0(H)$.
\end{lemma}

\begin{lemma} \label{lem:quotient-isom}
    Let $(\tilde{M},\tilde{g})$ be a simply-connected Riemannian manifold, and $\Gamma \subset \I(\tilde{M},\tilde{g})$ a finite subgroup of isometries acting freely. Then the quotient manifold $(M,g)$ has isometry group
    \begin{align}
        \I(M,g) = N_{\I(\tilde{M},\tilde{g})}(\Gamma)/\Gamma.
    \end{align}
\end{lemma}

If $\I(M,g)$ acts transitively on $(M,g)$, then we call $g$ a \emph{homogeneous} metric. Finally, if a compact Lie group $G$ acts smoothly on $M$, then the \emph{cohomogeneity} of the action is the minimal codimension of an orbit.

\subsection{Eguchi--Hanson space} \label{subsec:EH}

Let $\{X_1,X_2,X_3\}$ be a left-invariant frame on $\Sph^3 \cong SU(2)$, and $\{\sigma_1,\sigma_2,\sigma_3\}$ its dual coframe. Let $\pi: \Sph^3 \to \Sph^2$ be the Hopf fibration, chosen with fibers tangent to $X_1$. For each $\epsilon > 0$, define a metric $\tilde{g}_{\mathrm{EH},\epsilon}$ on the manifold $(0,\infty) \times \Sph^3$ by
\begin{align} \label{eq:EH-cover-metric}
    \tilde{g}_{\mathrm{EH},\epsilon} = \frac{1}{2}(r^2 + \epsilon)^{-\frac{1}{2}} (dr^2 + r^2 \sigma_1 \otimes \sigma_1) + (r^2 + \epsilon)^{\frac{1}{2}} (\sigma_2 \otimes \sigma_2 + \sigma_3 \otimes \sigma_3).
\end{align}
The metric $\tilde{g}_{\mathrm{EH},\epsilon}$ is invariant under left multiplication in $SU(2)$ and a right Hopf circle action; see Lemma \ref{lem:homog-metrics-S3}. Moreover, since the metric on each $r$-level set depends on $r$, there are no translational isometries. Thus, we have
\begin{align} \label{eq:EH-cover}
    \I^0((0,\infty) \times \Sph^3, \tilde{g}_{\mathrm{EH},\epsilon}) = SU(2) \times_{\Z_2} U(1) \cong U(2).
\end{align}
This acts with cohomogeneity one since it acts transitively on each level set of $r$.

In particular, $\tilde{g}_{\mathrm{EH},\epsilon}$ is invariant under the antipodal map on $\Sph^3$ (equivalently $q \mapsto -q$ on $SU(2)$), so it induces a metric $g_{\mathrm{EH},\epsilon}$ on $(0, \infty) \times \RP^3$. As $r \to 0$, the Hopf fibers shrink and the $\RP^3$-level sets of $r$ collapse to a round $\Sph^2$ of scale $\epsilon > 0$. Thus, attaching an $\Sph^2 \cong \CP^1$ `bolt' at $r=0$ results in a complete metric $g_{\mathrm{EH},\epsilon}$ on the smooth manifold $\O_{\CP^1}(-2) \cong T^*\Sph^2$. We call $(\O_{\CP^1}(-2), g_{\mathrm{EH},\epsilon})$ the \emph{$\epsilon$-Eguchi--Hanson space}. Changing $\epsilon$ only amounts to scaling the metric and applying a diffeomorphism, so we leave $\epsilon$ unspecified.

Since the antipodal map is central in the group \eqref{eq:EH-cover}, Lemma \ref{lem:quotient-isom} implies that $\I^0((0,\infty) \times \RP^3, g_{\mathrm{EH},\epsilon}) \cong U(2)/\Z_2$. These isometries extend to the $\Sph^2$ bolt, so $\I^0(\O_{\CP^1}(-2), g_{\mathrm{EH},\epsilon}) \cong U(2)/\Z_2$ as well. Thus, $U(2)$ acts isometrically and almost effectively on Eguchi--Hanson space with kernel $\Z_2$. Moreover, one checks using \eqref{eq:EH-cover-metric} that Eguchi--Hanson space is Ricci-flat and ALE with group at infinity $\Z_2$. It is in fact hyperk\"ahler and self-dual.

It is well-known that $\ker_{L^2}(\Delta_L)$ has dimension 3 on Eguchi--Hanson space; see \cites{biquard,brendle-kapouleas} for an explicit basis. One generator arises from scaling, and the other two generators arise from deformations of the complex structure. Thus, modulo diffeomorphisms and scaling, these generators do not produce genuinely new Einstein deformations.

\subsection{Calabi space} \label{subsec:calabi}

Let $\pi: \Sph^{2m-1} \to \CP^{m-1}$ be the Hopf fibration, where $m \geq 2$. Let $g_{\mathrm{FS}}$ be the Fubini--Study metric on $\CP^{m-1}$ with holomorphic sectional curvature 4, so that the round metric on $\Sph^{2m-1}$ of sectional curvature 1 splits as
\begin{align}
    g_{\Sph^{2m-1}} = \sigma \otimes \sigma + \pi^*g_{\mathrm{FS}},
\end{align}
where $\sigma$ is a 1-form dual to the Hopf vector field $\xi$, satisfying $d\sigma = 2\pi^*\omega_{\mathrm{FS}}$. For any $a, b > 0$, define the \emph{Berger metric} on $\Sph^{2m-1}$ by
\begin{align}
    \tilde{g}_{a,b} = a^2 \sigma \otimes \sigma + b^2 \pi^*g_{\mathrm{FS}}.
\end{align}
It is standard that $\I^0(\Sph^{2m-1}, \tilde{g}_{a,b}) \cong U(m)$ if $a \neq b$, and $\I^0(\Sph^{2m-1}, \tilde{g}_{a,b}) \cong SO(2m)$ if $a = b$.

The cyclic subgroup $\Z_m \subset U(1) \subset U(m)$ acts freely on $\Sph^{2m-1}$ by rotation of the Hopf fibers. The quotient is the \emph{lens space} which we denote by $$\Sph^{2m-1}/\Z_m.$$
Moreover, the $\Z_m$-action is isometric on $(\Sph^{2m-1},\tilde{g}_{a,b})$ for any $a,b > 0$, and is central in the isometry group, so $\tilde{g}_{a,b}$ descends to a \emph{quotient-Berger metric} $g_{a,b}$ on $\Sph^{2m-1}/\Z_m$. By Lemma \ref{lem:quotient-isom}, we have $\I^0(\Sph^{2m-1}/\Z_m, g_{a,b}) \cong U(m)/\Z_m$ whenever $a \neq b$.

Consider the complex line bundle $\O_{\CP^{m-1}}(-m)$ over $\CP^{m-1}$, which is isomorphic to the canonical line bundle of $\CP^{m-1}$. We now briefly review Calabi's construction \cite{calabi} of Ricci-flat K\"ahler ALE metrics on $\O_{\CP^{m-1}}(-m)$ using what is now called the \emph{Calabi ansatz}; a detailed account may be found elsewhere (e.g. \cite{sal}*{\S 8}). We first note that $\O_{\CP^{m-1}}(-m)$ has sphere bundles diffeomorphic to $\Sph^{2m-1}/\Z_m$. Thus, choosing a radial coordinate $r$ along the fibers, we obtain a diffeomorphism
\begin{align} \label{eq:bundle-diffeo}
    \O_{\CP^{m-1}}(-m) \setminus \CP^{m-1} \cong (0,\infty) \times \Sph^{2m-1}/\Z_m
\end{align}
where the excluded $\CP^{m-1}$ is the zero section. On $(0,\infty) \times \Sph^{2m-1}/\Z_m$, one considers metrics of the form
\begin{align} \label{eq:g-calabi}
    g = u(r) \left( dr^2 + r^2 \sigma \otimes \sigma \right) + v(r) \pi^*g_{\mathrm{FS}} = u(r) dr^2 + g_{r\sqrt{u(r)}, \sqrt{v(r)}},
\end{align}
where $u$ and $v$ are smooth positive functions. Equivalently, with the canonical complex structure under which $J(\del_r) = \frac{1}{r}\xi$, one considers K\"ahler forms on $(0,\infty) \times \Sph^{2m-1}/\Z_m$ of the form
\begin{align} \label{eq:calabi-kahler}
    \omega = u(r)r \, dr \wedge \sigma + v(r) \pi^*\omega_{\mathrm{FS}},
\end{align}
where $\omega_{\mathrm{FS}}$ is the Fubini--Study K\"ahler form on $\CP^{m-1}$. The K\"ahler condition $d\omega = 0$ and the Ricci-flatness condition yield, respectively,
\begin{align} \label{eq:ODE-conditions}
    v' = 2ru, \quad uv^{m-1} = 2.
\end{align}
Solving the system \eqref{eq:ODE-conditions}, we get the general solution
\begin{align} \label{eq:uv-solution}
    u(r) = 2(2m(r^2+\epsilon))^{-\frac{m-1}{m}}, \quad v(r) = (2m(r^2+\epsilon))^{\frac{1}{m}}
\end{align}
where $\epsilon \in \R$. Let $g_{\mathrm{Cal},\epsilon}$ be the metric \eqref{eq:g-calabi} with $u$ and $v$ given by \eqref{eq:uv-solution}. If $\epsilon > 0$, then $g_{\mathrm{Cal},\epsilon}$ closes up smoothly at $r=0$ to give a complete metric on $\O_{\CP^{m-1}}(-m)$, also denoted by $g_{\mathrm{Cal},\epsilon}$. We call $(\O_{\CP^{m-1}}(-m), g_{\mathrm{Cal},\varepsilon})$ the \emph{$\varepsilon$-Calabi space}. Changing $\varepsilon > 0$ only amounts to scaling the metric and applying a diffeomorphism, so we leave $\varepsilon$ unspecified.

It follows from the above discussion about isometry groups of Berger metrics that $\I^0(\O_{\CP^{m-1}}(-m), g_{\mathrm{Cal},\varepsilon}) \cong U(m)/\Z_m$. Thus, $U(m)$ acts isometrically and almost effectively on Calabi space with kernel $\Z_m$. The action preserves and is transitive on each level set of $r$, so it has cohomogeneity one. Moreover, as one checks using \eqref{eq:g-calabi} and \eqref{eq:uv-solution}, Calabi space is Ricci-flat and K\"ahler ALE with group at infinity $\Z_m$. When $m=2$, Calabi space coincides with Eguchi--Hanson space and is therefore hyperk\"ahler. On the other hand, the Calabi spaces with $m \geq 3$ are not hyperk\"ahler.

Observe that if $h \in \ker_{L^2}(\Delta_L)$, then
\begin{align}
    \Delta \tr_g h &= \tr_g \Delta h = -2 \tr_g \Rm(h) = -2 g^{ik} \Rm_{ijkl} h^{jl} = 2 \Ric_{jl} h^{jl} = 0.
\end{align}
Thus, $\tr_g h \in \ker_{L^2}(\Delta) = \{0\}$. It then follows by \cite{morteza-viaclovsky}*{Theorem 2.2} that $\ker_{L^2}(\Delta_L)$ is one-dimensional on each Calabi space with $m \geq 3$. The generator of $\ker_{L^2}(\Delta_L)$ arises from scaling, so it does not produce a genuinely new Einstein deformation.

\section{Global symmetry principles} \label{sec:symmetry-principles}

\subsection{Geometry of ALE spaces}

First, we have the following result about decay rates of Ricci-flat ALE spaces, due to Kr\"oncke and Sz\'abo (see also \cites{bkn89,cheeger-tian}). The original proof is given for Ricci-flat ALE manifolds, but it applies to Ricci-flat ALE spaces after standard orbifold adaptations of the weighted elliptic theory.

\begin{theorem}[\cite{kroncke-szabo}] \label{thm:ALE-decay-rate}
    Let $(M^n,g)$ be an $n$-dimensional Ricci-flat ALE space with group at infinity $\Gamma \neq \{1\}$. Then $(M,g)$ is of order $n$.
\end{theorem}

We will also use that every ALE space with nontrivial group at infinity admits an end foliation by constant mean curvature (CMC) hypersurfaces. This follows from work of Chodosh, Eichmair, and Volkmann.

\begin{theorem}[\cite{cev}] \label{thm:cmc-foliation}
    Let $(M^n,g)$ be an ALE space with group at infinity $\Gamma \neq \{1\}$, and let $\mathcal{S} = \Sph^{n-1}/\Gamma$. Then there exists $R_0 \gg 1$ and a smooth family of hypersurfaces $\{\Sigma_\rho\}_{\rho \geq R_0}$ foliating the end of $M$, such that:
    \begin{enumerate}[label=(\alph*)]
        \item $\Sigma_\rho$ is diffeomorphic to $\mathcal{S}$ and has constant mean curvature $\frac{n-1}{\rho}$.
        \item In any ALE chart $\Phi: M \setminus K \to (\R^n \setminus B_R(0))/\Gamma \cong (R,\infty) \times \mathcal{S}$, we have for each $\rho \geq R_0$
        \begin{align} \label{eq:Sigma-identification}
            \Phi(\Sigma_\rho) = \{ (\rho(1+u_\rho(\theta)), \theta) \mid \theta \in \mathcal{S} \}
        \end{align}
        where $u_\rho \in C^{2,\alpha}(\mathcal{S})$ satisfies $\norm{u}_{C^{2,\alpha}(\mathcal{S})} \to 0$ as $\rho \to \infty$. In fact, by elliptic bootstrapping, $\norm{u_\rho}_{C^k(\mathcal{S})} \to 0$ for every $k \geq 0$.
        \item Let $H$ be the function defined on the end of $M$ by sending a point $x$ to the mean curvature of the CMC leaf containing $x$. Then $H$ is smooth and $\nabla H$ is nonvanishing.
    \end{enumerate}
\end{theorem}

In the rest of this subsection, we establish the following result which exhibits a relation between Killing fields and the CMC leaves.

\begin{corollary} \label{cor:KF-tangential}
    Let $(M,g)$ be a Ricci-flat ALE space with group at infinity $\Gamma \neq \{1\}$. If $W$ is a Killing field on $(M,g)$, then $W$ is tangential to all sufficiently large leaves of the CMC foliation from Theorem \ref{thm:cmc-foliation}, i.e. $\inner{W}{\nabla H} = 0$. Moreover, $[W,\nabla H] = 0$ on such leaves.
\end{corollary}

\begin{lemma} \label{lem:g-A-convergence}
    Let $(M,g)$ be an ALE space with group at infinity $\Gamma \neq \{1\}$.
    Let $g_\rho$ be the induced metric on the CMC leaf $\Sigma_\rho$ with mean curvature $\frac{n-1}{\rho}$, and $A_\rho$ the second fundamental form of $\Sigma_\rho$. Using \eqref{eq:Sigma-identification}, we identify $\Sigma_\rho$ with $\mathcal{S}$, and consider $g_\rho$ and $A_\rho$ as symmetric 2-tensors on $\mathcal{S}$. Then as $\rho \to \infty$, we have
    \begin{align}
        \rho^{-2} g_\rho \to \bar{g}_\infty \quad \text{and} \quad \rho^2 |A_\rho|_{g_\rho}^2 \to n-1 \quad \text{smoothly on } \mathcal{S},
    \end{align}
    where $\bar{g}_\infty$ is a metric on $\mathcal{S}$ with constant sectional curvature 1.
\end{lemma}
\begin{proof}
    By Theorem \ref{thm:cmc-foliation}, after identifying $\Sigma_\rho$ with $\mathcal{S}$ as described and rescaling the metric on $\Sigma_\rho$ by $\rho^{-2}$, the resulting metric converges smoothly to that of the unit slice $\{r=1\}$ in the cone $((0,\infty) \times \mathcal{S}, dr^2 + r^2 \bar{g}_\infty)$. Hence $\rho^{-2} g_\rho \to \bar{g}_\infty$ smoothly. Likewise, the rescaled second fundamental form converges smoothly to the second fundamental form of $\{r=1\}$, which has squared norm $n-1$. Therefore $\rho^2 |A_\rho|_{g_\rho}^2 = |\rho^{-1} A_\rho|_{\rho^{-2} g_\rho}^2 \to n-1$ smoothly.
\end{proof}

Recall that if $\Gamma$ is nontrivial, then the first nonzero eigenvalue of $-\Delta_{\bar{g}_\infty}$ on functions $\Sph^{n-1}/\Gamma \to \R$ is strictly larger than $n-1$. Using this fact and a standard compactness--contradiction argument, we obtain the next lemma.

\begin{lemma} \label{lem:convergence-eigenfn}
    Let $\Gamma \neq \{1\}$ and $\mathcal{S} = \Sph^{n-1}/\Gamma$. There exists $\epsilon > 0$ such that if $g$ is a metric on $\mathcal{S}$ and $\phi$ is a smooth function on $\mathcal{S}$ such that
    \begin{align}
        \norm{g - \bar{g}_\infty}_{C^{2,\alpha}(\bar{g}_\infty)} < \epsilon, \quad \norm{\phi - (n-1)}_{C^{0,\alpha}(\bar{g}_\infty)} < \epsilon,
    \end{align}
    then $\ker(\Delta_g + \phi) = 0$.
\end{lemma}
\begin{proof}
    If the lemma were false, then there exist sequences $g_k \to \bar{g}_\infty$ in $C^{2,\alpha}(\bar{g}_\infty)$ and $\phi_k \to n-1$ in $C^{0,\alpha}(\bar{g}_\infty)$, as well as nonzero functions $u_k$ such that
    \begin{align} \label{eq:6018608}
        (\Delta_{g_k} + \phi_k) u_k = 0.
    \end{align}
    Normalizing, we may assume $\norm{u_k}_{L^2(\bar{g}_\infty)} = 1$. Standard elliptic estimates and the $C^{2,\alpha}$-convergence $g_k \to \bar{g}_\infty$ imply a uniform bound
    \begin{align}
        \norm{u_k}_{H^2(\bar{g}_\infty)} \leq C\left( \norm{\Delta_{g_k} u_k}_{L^2(\bar{g}_\infty)} + \norm{u_k}_{L^2(\bar{g}_\infty)} \right) \leq C.
    \end{align}
    We can then pass to a subsequence and obtain a limit $u_k \to u$ in $H^1(\bar{g}_\infty)$, with $\norm{u}_{L^2(\bar{g}_\infty)} = 1$. Now, \eqref{eq:6018608} implies that for any smooth function $\eta$ on $\mathcal{S}$,
    \begin{align}
        -\int_{\mathcal{S}} \inner{\nabla^{g_k} \eta}{\nabla^{g_k} u_k} \, \dvol_{g_k} + \int_{\mathcal{S}} \phi_k u_k \eta \, \dvol_{g_k} = 0.
    \end{align}
    Taking $k \to \infty$, we get
    \begin{align}
        -\int_{\mathcal{S}} \inner{\nabla^{\bar{g}_\infty} \eta}{\nabla^{\bar{g}_\infty} u} \, \dvol_{\bar{g}_\infty} + \int_{\mathcal{S}} (n-1) u \eta \, \dvol_{\bar{g}_\infty} = 0,
    \end{align}
    so $(\Delta_{\bar{g}_\infty} + (n-1))u = 0$ weakly. By elliptic regularity, $u$ is smooth and the PDE holds classically. Since the first nonzero eigenvalue of $-\Delta_{\bar{g}_\infty}$ on functions $\mathcal{S} \to \R$ is strictly larger than $n-1$, we must have $u=0$. But this contradicts the fact that $\norm{u}_{L^2(\bar{g}_\infty)} = 1$.
\end{proof}

\begin{proof}[Proof of Corollary \ref{cor:KF-tangential}]
    Let $\Sigma_\rho \subset M$ be a CMC leaf, $g_\rho$ and $A_\rho$ its induced metric and second fundamental form, and $H$ the mean curvature function as in Theorem \ref{thm:cmc-foliation}. Let $\Sigma_{\rho,t}$ be the image of $\Sigma_\rho$ under the time $t$ flow of $W$, and $H_{\rho,t}$ its mean curvature. By the first variation of mean curvature on a CMC hypersurface in a Ricci-flat manifold (e.g. \cite{leegeometric}*{p. 32}), we have at each point on $\Sigma_\rho$
    \begin{align} \label{eq:first-var-H}
        \frac{d}{dt}\Big|_{t=0} H_{\rho,t} = -\Delta_{g_\rho} \inner{W}{\nu} - |A_\rho|^2 \inner{W}{\nu}
    \end{align}
    where $\nu = -\frac{\nabla H}{|\nabla H|}$ is the outward unit normal to $\Sigma_\rho$. But $W$ is a Killing field, so $H_{\rho,t} = \frac{n-1}{\rho}$ for all $t$. Combining with \eqref{eq:first-var-H} yields
    \begin{align}
        \Delta_{g_\rho} \inner{W}{\nu} + |A_\rho|^2 \inner{W}{\nu} = 0.
    \end{align}
    Multiplying both sides of this equation by $\rho^2$ gives
    \begin{align}
        \Delta_{\rho^{-2}g_\rho} \inner{W}{\nu} + \rho^2 |A_\rho|_g^2 \inner{W}{\nu} = 0.
    \end{align}
    By Lemma \ref{lem:g-A-convergence}, $\rho^{-2}g_\rho \to \bar{g}_\infty$ and $\rho^2 |A_\rho|_g^2 \to n-1$ smoothly as $\rho \to \infty$. Then by Lemma \ref{lem:convergence-eigenfn}, $\inner{W}{\nu} \equiv 0$ and hence $\inner{W}{\nabla H} = 0$ whenever $\rho$ is sufficiently large. This establishes the first claim in the corollary.

    Next, we recall a general identity for vector fields $X$ and $Y$:
    \begin{align}
        \L_X(Y^\flat) = (\L_X Y)^\flat + (\L_X g)(Y,\cdot).
    \end{align}
    Applying this with $X=W$ and $Y=\nabla H$ gives
    \begin{align}
        [W,\nabla H] = \L_W(\nabla H) = (\L_W dH)^\sharp = (d\L_W H)^\sharp = (d\inner{W}{\nabla H})^\sharp = 0,
    \end{align}
    which completes the proof.
\end{proof}

\subsection{Finding global Killing fields}

For the rest of this section, let $(M^n,g)$ be a complete Ricci-flat ALE space of dimension $n \geq 4$ with group at infinity $\Gamma \neq \{1\}$. Write $\mathcal{S} = \Sph^{n-1}/\Gamma$, and $\bar{g}_\infty$ for the round metric on $\mathcal{S}$ with constant sectional curvature 1. Using Theorem \ref{thm:ALE-decay-rate}, we fix a diffeomorphism $\Phi: M \setminus K \to (\R^n \setminus B_R(0))/\Gamma$ such that for each $k \geq 0$,
\begin{align} \label{eq:ALE-diffeo-order-n}
    |\nabla^k(\Phi^*\bar{g} - g)|_g \leq \O(r^{-n-k}),
\end{align}
where $\bar{g}$ is the Euclidean metric on $(\R^n \setminus B_R(0))/\Gamma$.
Given a vector field $\bar{U}$ on $\mathcal{S}$, we can radially extend it to a linearly growing vector field on $(\R^n \setminus B_R(0))/\Gamma$ also denoted by $\bar{U}$. We will say that a vector field $U$ on $M$ \emph{extends} $\bar{U}$ if $\Phi_*U = \bar{U}$ on $(\R^n \setminus B_R(0))/\Gamma$.

\begin{proposition} \label{prop:approx-KF}
    Let $\bar{U}$ be a Killing field on $(\mathcal{S},\bar{g}_\infty)$, and let $U$ be a smooth vector field on $M$ which extends $\bar{U}$.
    Then there exist a smooth vector field $W$ on $M$ and $\epsilon \in (0,\frac{1}{2})$ such that
    \begin{align}
        \Delta W &= 0, \label{eq:W1} \\
        |\nabla^k(W-U)| &\leq \O(r^{-n+2+\epsilon-k}) \text{ for each } k \geq 0, \label{eq:W2} \\
        |\nabla^k W| &\leq \O(r^{1-k}) \text{ for each } k \geq 0, \label{eq:W3} \\
        |\L_W g| &\leq \O(r^{-n+1+\epsilon}). \label{eq:W4}
    \end{align}
\end{proposition}
\begin{proof}
    Let $h = \Phi^*\bar{g} - g$. The vector field $\bar{U}$ on $(\R^n \setminus B_R(0))/\Gamma$ satisfies $|\bar{\nabla}^k \bar{U}|_{\bar{g}} \leq \O(r^{1-k})$ for each $k \geq 0$, where $\bar{\nabla}$ is the Euclidean derivative. Using \eqref{eq:ALE-diffeo-order-n}, one can deduce
    \begin{align} \label{eq:nabla-k-U}
        |\nabla^k U| \leq \O(r^{1-k}).
    \end{align}
    Using \eqref{eq:ALE-diffeo-order-n} and \eqref{eq:nabla-k-U}, we also have on $M \setminus K$ that
    \begin{equation} \label{eq:nablakLuGbound}
        |\nabla^k \L_U g| = |\nabla^k \L_U h| = |\nabla^k (\nabla U \ast h + U \ast \nabla h)| \leq \O(r^{-n-k}) \text{ for each } k \geq 0.
    \end{equation}
    Next, recall the general relation
    \begin{align}
        \div(\L_U g) - \frac{1}{2} \nabla (\tr \L_U g) = \Delta U + \Ric(U).
    \end{align}
    By \eqref{eq:nablakLuGbound} and Ricci-flatness, this implies
    \begin{align} \label{eq:Delta-Ui-decay}
        |\nabla^k \Delta U| \leq C|\nabla^{k+1} \L_U g| \leq \O(r^{-n-1-k}) \quad \text{for all } k \geq 0.
    \end{align}
    For $k \geq 0$, $\alpha \in (0,1)$, and $\beta \in \R$, let $C^{k,\alpha}_\beta$ be the weighted H\"older space of vector fields $X$ on $M$ such that
    \begin{align}
        \norm{X}_{C^{k,\alpha}_\beta} := \sum_{j=0}^k \left( \sup_{x \in M} r(x)^{j-\beta} |\nabla^j X(x)| \right) + \sup_{x \in M} r(x)^{k+\alpha-\beta} [\nabla^k X]_{C^\alpha(B_{r(x)/2}(x))}
    \end{align}
    is finite. By standard elliptic theory on ALE spaces (see for instance \cite{kroncke-szabo}*{\S 4.1}), we have that for all $k \in \mathbb{N}_0$, $\alpha \in (0,1)$, and all $\beta \in \R$ outside some discrete set $\mathcal{D} \subset \R$,
    \begin{align} \label{eq:Delta-mapping}
        \Delta: C^{k+2,\alpha}_{\beta+2} \to C^{k,\alpha}_{\beta}
    \end{align}
    is Fredholm with image
    \begin{align}
        \left\{ Y \in C^{k,\alpha}_\beta : \int_M \inner{Y}{X} = 0 \text{ for all } X \in \mathcal{H}_{-\beta-n} \right\},
    \end{align}
    where $\mathcal{H}_{-\beta-n}$ is the space of harmonic vector fields in $C^{\ell,\alpha}_{-\beta-n}$ for arbitrary $\ell$. We let $\beta = -n+\epsilon$, where $\epsilon \in (0,\frac{1}{2})$ is chosen so that $\beta \notin \mathcal{D}$. Then any $X \in \mathcal{H}_{-\beta-n}$ satisfies $|X| \leq \O(r^{-\epsilon})$ and is therefore in $L^p$ for some $p \in (1,\infty)$. By \cite{kroncke-petersen}*{Proposition 4.3}, this promotes to $|X| \leq \O(r^{-n+1})$. Thus $X$ is an $L^2$ harmonic vector field, which must be zero by the Bochner formula and an integration-by-parts argument. Hence, \eqref{eq:Delta-mapping} is surjective for $\beta = -n+\epsilon$.
    
    By \eqref{eq:Delta-Ui-decay}, we have $-\Delta U \in C^{k,\alpha}_{-n+\epsilon}$ for all $k \geq 0$ and $\alpha \in (0,1)$. Hence there is a vector field $\hat{U} \in C^{2,\alpha}_{-n+2+\epsilon}$ such that
    \begin{align}
        \Delta \hat{U} = -\Delta U.
    \end{align}
    Standard weighted elliptic estimates yield $\hat{U} \in C^{k+2,\alpha}_{-n+2+\epsilon}$ for all $k \geq 0$. So $W := U + \hat{U}$ satisfies \eqref{eq:W1} and \eqref{eq:W2}. Combining with \eqref{eq:nabla-k-U} and \eqref{eq:nablakLuGbound}, we obtain \eqref{eq:W3} and \eqref{eq:W4} respectively.
\end{proof}

Recall that the Lichnerowicz Laplacian $\Delta_L$ is defined by \eqref{eq:lichnerowicz-laplacian} on a Ricci-flat space. The following identity is well-known: see for instance \cite{brendle2013rotational}*{Theorem 4.1} (taking $X=0$ there).
\begin{lemma} \label{lem:VF-L-identity}
    On a Ricci-flat space $(M,g)$, if $W$ is a vector field satisfying $\Delta W = 0$, then $\Delta_L(\L_W g) = 0$.
\end{lemma}

Next, we show that a dimension bound on $\ker_{L^2}(\Delta_L)$ implies the existence of Killing fields on $(M,g)$. This is a direct analog of \cite{law}*{Proposition 3.8} for Ricci-flat ALE spaces. Recall that $\I(M,g)$ denotes the isometry group of $(M,g)$, $\I^0(M,g)$ is its identity component, and $\mathfrak{I}(M,g)$ is its Lie algebra (i.e. the Lie algebra of Killing fields). Define $\bar{d} = \dim \I(\mathcal{S},\bar{g}_\infty)$.

\begin{proposition} \label{prop:exact-KFs}
    Suppose $\dim\ker_{L^2}(\Delta_L) \leq d$ for some nonnegative integer $d$. Then there exists a $(\bar{d}-d)$-dimensional vector space $\mathcal{V}$ of Killing fields on $(M,g)$ such that for each $W \in \mathcal{V}$,
    \begin{enumerate}[label=(\alph*)]
        \item $[W,\nabla H] = 0$.
        \item $\inner{W}{\nabla H} = 0$, i.e. $W$ is tangent to the CMC leaves.
        \item $W$ is a Killing field for $(M,g)$ and for each sufficiently far CMC leaf.
        \item $|\nabla^k W| \leq \O(r^{1-k})$ for each $k \geq 0$.
        \item There exists a unique Killing field $\bar{U}$ on $(\mathcal{S},\bar{g}_\infty)$ such that if $U$ is a smooth vector field on $M$ which extends $\bar{U}$, then $|\nabla^k (W-U)| \leq \O(r^{-n+2+\epsilon-k})$ for each $k \geq 0$. The map $\mathcal{V} \to \mathfrak{I}(\mathcal{S},\bar{g}_\infty)$, $W \mapsto \bar{U}$ is linear and injective.
    \end{enumerate}
\end{proposition}
\begin{proof}
    Let $\{\bar{U}_i\}_{i=1}^{\bar{d}}$ be a basis for $\mathfrak{I}(\mathcal{S},\bar{g}_\infty)$, and let $\{U_i\}_{i=1}^{\bar{d}}$ be smooth vector fields on $M$ which extend $\bar{U}_i$. Proposition \ref{prop:approx-KF} gives vector fields $\{W_i\}_{i=1}^{\bar{d}}$ and $\epsilon \in (0,\frac{1}{2})$ such that
    \begin{align}
        \Delta W_i &= 0, \label{eq:Deltaf-W} \\
        |\nabla^k(W_i-U_i)| &\leq \O(r^{-n+2+\epsilon-k}) \text{ for each } k \geq 0, \label{eq:O1-bound} \\
        |\nabla^k W_i| &\leq \O(r^{1-k}) \text{ for each } k \geq 0, \label{eq:kWi-bound} \\
        |\L_{W_i} g| &\leq \O(r^{-n+1+\epsilon}). \label{eq:LWg-estimate}
    \end{align}
    By \eqref{eq:Deltaf-W}, the decay \eqref{eq:LWg-estimate}, and Lemma \ref{lem:VF-L-identity}, we have $\L_{W_i} g \in \ker_{L^2}(\Delta_L)$. Since the vector fields $U_i$ are linearly independent, \eqref{eq:O1-bound} implies that the vector fields $W_i$ are linearly independent too.

    Let $\mathcal{V}' = \mathrm{span}\{W_i\}_{i=1}^{\bar{d}}$, so that $\dim\mathcal{V}' = \bar{d}$ by the linear independence. Define a map $\mathcal{R}: \mathcal{V}' \to \mathfrak{I}(\mathcal{S},\bar{g}_\infty)$ by sending $W_i$ to $\bar{U}_i$ and extending linearly. This map is clearly injective. Also define a map $\mathcal{T}: \mathcal{V}' \to \ker_{L^2}(\Delta_L)$ by $W \mapsto \L_W g$.

    The assumption $\dim\ker_{L^2}(\Delta_L) \leq d$ implies $\dim \ker \mathcal{T} \geq \bar{d}-d$. Let $\mathcal{V} = \ker\mathcal{T}$. Then every $W \in \mathcal{V}$ is a Killing field for $(M,g)$. Combined with Corollary \ref{cor:KF-tangential}, this proves (a), (b), and (c). Part (d) follows from \eqref{eq:kWi-bound}. Part (e) follows from restricting the map $\mathcal{R}: \mathcal{V}' \to \mathfrak{I}(\mathcal{S},\bar{g}_\infty)$ to $\mathcal{V}$, and applying \eqref{eq:O1-bound}.
\end{proof}

For all large $\rho > 0$, let $\Sigma_\rho \subset M$ be the CMC leaf with mean curvature $H = \frac{n-1}{\rho}$. Using Theorem \ref{thm:cmc-foliation}(b), we identify $\Sigma_\rho$ with a Riemannian manifold $(\mathcal{S}, g_\rho)$. We set up the following:
\begin{itemize}
    \item Let $\mathcal{V}_\rho$ be the restriction of Killing fields in $\mathcal{V}$ to $\Sigma_\rho$. By Proposition \ref{prop:exact-KFs}, $\mathcal{V}_\rho \subset \mathfrak{I}(\mathcal{S},g_\rho)$.
    \item Let $\mathfrak{g}$ be the Lie subalgebra of $\mathfrak{I}(M,g)$ generated by $\mathcal{V}$, and let $G$ be the unique connected Lie subgroup of $\I^0(M,g)$ generated by $\mathfrak{g}$.
    \item Likewise, let $\mathfrak{g}_\rho$ be the Lie subalgebra of $\mathfrak{I}(\mathcal{S},g_\rho)$ generated by $\mathcal{V}_\rho$, and let $G_\rho$ be the unique connected Lie subgroup of $\I^0(\mathcal{S},g_\rho)$ generated by $\mathfrak{g}_\rho$.
\end{itemize} 
The next corollary immediately follows from Corollary \ref{cor:KF-tangential}.
\begin{corollary} \label{cor:g-jacobi}
    For any $X \in \mathfrak{g}$, we have $[X,\nabla H] = 0$ and $\inner{X}{\nabla H} = 0$.
\end{corollary}

With the help of Corollary \ref{cor:g-jacobi}, the following assertions are proved in exactly the same manner as \cite{law}*{Lemma 3.19}. Thus we omit the details.
\begin{lemma} \label{lem:level-set-isos}
    \begin{enumerate}[label=(\alph*)]
        \item The restriction map $\mathcal{V} \to \mathcal{V}_\rho$ is a linear isomorphism.
        \item The restriction map $\mathfrak{g} \to \mathfrak{g}_\rho$ is a Lie algebra isomorphism.
        \item If $G_\rho$ is closed in $\I(\mathcal{S}, g_\rho)$, then $G$ is a compact Lie group acting smoothly and isometrically on $(M,g)$, and the restriction map $G \to G_\rho$ is a Lie group isomorphism.
    \end{enumerate}
\end{lemma}

In the case where $(\mathcal{S},\bar{g}_\infty)$ is homogeneous, i.e. its isometry group acts transitively, the space of Killing fields $\mathfrak{I}(\mathcal{S},\bar{g}_\infty)$ spans the tangent space to $\mathcal{S}$ at each point. We define the quantity $d_1(\mathcal{S}, \bar{g}_\infty) \geq 0$ to be the largest integer $d$ such that
\begin{itemize}
    \item Any Lie subalgebra of $\mathfrak{I}(\mathcal{S},\bar{g}_\infty)$ with dimension $\geq \bar{d} - d$ spans the tangent space to $\mathcal{S}$ at every point.
\end{itemize}

\begin{lemma} \label{lem:exact-KF-subalgebra}
    If $(\mathcal{S},\bar{g}_\infty)$ is homogeneous and $\dim\ker_{L^2}(\Delta_L) \leq d_1(\mathcal{S},\bar{g}_\infty)$, then $\mathfrak{g}_\rho$ spans the tangent space to $\Sigma_\rho$ at every point for all sufficiently large $\rho$.
\end{lemma}
\begin{proof}
    Proposition \ref{prop:exact-KFs} gives a $(\bar{d}-d_1(\mathcal{S}, \bar{g}_\infty))$-dimensional space $\mathcal{V}$ of Killing fields on $(M,g)$. Take a basis $\{W_i\}_{i=1,\ldots,\bar{d}-d_1(\mathcal{S}, \bar{g}_\infty)}$ for $\mathcal{V}$, let $\bar{U}_i$ be the corresponding vector fields from Proposition \ref{prop:exact-KFs}(e), and let $U_i$ be arbitrary extensions thereof. Since the map $\mathcal{V} \to \mathfrak{I}(\mathcal{S},\bar{g}_\infty)$, $W_i \mapsto \bar{U}_i$ is linear and injective, it follows that the $U_i$ are linearly independent.

    Write $X_i = W_i - U_i$. From Propositions \ref{prop:approx-KF} and \ref{prop:exact-KFs}, we have
    \begin{align} \label{eq:WXU-estimates}
        |\nabla^k W_i| \leq \O(r^{1-k}), \quad |\nabla^k X_i| \leq \O(r^{-n+2+\epsilon-k}), \quad |\nabla^k U_i| \leq \O(r^{1-k})
    \end{align}
    for each $k \geq 0$.
    The rest is then identical to \cite{law}*{Lemma 3.21}, so we omit the details.
\end{proof}

The next proposition elucidates the global structure of $(M,g)$ if it admits an isometric action of cohomogeneity one. It is an orbifold version of \cite{law}*{Proposition 3.22}, which follows from the work of Mostert \cite{mostert} in the setting of smooth manifolds and Gonz\'alez \'Alvaro for smooth orbifolds \cite{gonzalez-alvaro-phd}.

\begin{proposition} \label{prop:coho-1-structure}
    Suppose there is a compact, connected Lie group $G$ acting smoothly, almost effectively, and isometrically with cohomogeneity one on $(M,g)$. Denote the projection map by $\pi: M \to M/G$. Then:
    \begin{enumerate}[label=(\alph*)]
        \item The orbit space $M/G$ is homeomorphic to $[0,1)$.
        \item There is a closed subgroup $H$ of $G$ such that the \emph{principal orbits}, i.e. $\pi^{-1}(s)$ for any $s > 0$, are diffeomorphic to $\mathcal{S} \cong G/H$. Thus, each point in a principal orbit has isotropy subgroup $H$.
        \item The \emph{singular orbit} $X = \pi^{-1}(0)$ is diffeomorphic to $G/K$, where $K$ is a closed subgroup of $G$ such that $H \subsetneq K$ and $K/H \cong \mathcal{S}'$ for some spherical space form $\mathcal{S}'$ of dimension $\ell \geq 0$.
        \item $M$ is diffeomorphic to the orbifiber bundle $E = G \times_K \mathrm{Cone}(\mathcal{S}')$ over $G/K \cong X$, where $K$ acts via $k \cdot (g,v) = (gk^{-1}, kv)$, and $kv$ is the radial extension of the action of $K$ on $K/H \cong \mathcal{S'}$.
        \item Every sphere bundle of $E$ is diffeomorphic to $\mathcal{S}$. Thus $M \setminus X$ is diffeomorphic to $(0,\infty) \times \mathcal{S}$, and under this identification the metric $g$ becomes
        \begin{align}
            g = ds^2 + g(s),
        \end{align}
        where $\{g(s)\}_{s>0}$ is a smooth family of $G$-invariant metrics on $G/H \cong \mathcal{S}$.
    \end{enumerate}
\end{proposition}

\section{A characterization of Eguchi--Hanson space} \label{sec:eh-result}

In this section, we prove Theorem \ref{thm:EH-characterization}. Before proceeding, we require some preparations regarding metrics on $\Sph^3$ and $\RP^3$ with large isometry groups.

\subsection{Metrics on $\Sph^3$ and $\RP^3$ with large isometry groups}

Below, we will identify $\Sph^3$ as the multiplicative group $SU(2)$ of unit quaternions
\begin{align}
    \Sph^3 \cong SU(2) = \{a + b\mathbf{i} + c\mathbf{j} + d\mathbf{k} \in \mathbb{H} : a^2 + b^2 + c^2 + d^2 = 1\}.
\end{align}
Any left-invariant metric $g$ on $SU(2)$ is of the form
\begin{align} \label{eq:left-invariant}
    \tilde{g} = \lambda^2 \sigma_1 \otimes \sigma_1 + \mu^2 \sigma_2 \otimes \sigma_2 + \nu^2 \sigma_3 \otimes \sigma_3,
\end{align}
where $\lambda,\mu,\nu > 0$ are constants, and $\{\sigma_1,\sigma_2,\sigma_3\}$ is dual to a left-invariant frame $\{X_1,X_2,X_3\}$. We may assume $X_1 = \mathbf{i}$, $X_2 = \mathbf{j}$, $X_3 = \mathbf{k}$ in $T_1 SU(2) = \mathfrak{su}(2)$. If $\lambda=\mu=\nu$, then $\tilde{g}$ is a round metric. If exactly two coefficients are equal, then we call $\tilde{g}$ a \emph{Berger metric}. If all three coefficients are different, then we call $\tilde{g}$ a \emph{generic left-invariant metric}. The corresponding isometry groups are well-known and are stated in the next lemma.
\begin{lemma} \label{lem:homog-metrics-S3}
    Let $\tilde{g}$ be a homogeneous metric on $\Sph^3$. Then $\tilde{g}$ is a left-invariant metric \eqref{eq:left-invariant} on $SU(2)$ \cite{sekigawa}*{Theorem C}. The isometry group $\I(SU(2), \tilde{g})$ is given as follows:
    \begin{enumerate}[label=(\alph*)]
        \item If $\tilde{g}$ is a round metric ($\lambda=\mu=\nu$), then $\I(SU(2), \tilde{g})$ is generated by left multiplications in $SU(2)$, right multiplications in $SU(2)$, and inversion $q \mapsto q^{-1}$ in $SU(2)$. Thus,
        \begin{align}
            \I(SU(2), \tilde{g}) = (SU(2)_L \times_{\Z_2} SU(2)_R) \rtimes \Z_2 \cong SO(4) \rtimes \Z_2 \cong O(4).
        \end{align}
        \item If $\tilde{g}$ is a Berger metric (without loss of generality we assume $\lambda \neq \mu = \nu$), then $\I(SU(2), \tilde{g})$ is generated by left multiplications in $SU(2)$, right multiplications by $e^{\mathbf{i}\theta}$ ($\theta \in [0,2\pi)$), and the map $q \mapsto \mathbf{j}^{-1} q\mathbf{j}$. Thus,
        \begin{align}
            \I(SU(2), \tilde{g}) = (SU(2)_L \times_{\Z_2} U(1)_R) \rtimes \Z_2 \cong U(2) \rtimes \Z_2.
        \end{align}
        \item If $\tilde{g}$ is a generic left-invariant metric, then $\I(SU(2), \tilde{g})$ is generated by left multiplications in $SU(2)$, the map $q \mapsto \mathbf{i}^{-1} q \mathbf{i}$, and the map $q \mapsto \mathbf{j}^{-1} q \mathbf{j}$. Thus,
        \begin{align}
            \I(SU(2), \tilde{g}) = SU(2)_L \rtimes (\Z_2)^2.
        \end{align}
    \end{enumerate}
\end{lemma}

Lemma \ref{lem:homog-metrics-S3} leads to a classification of homogeneous metrics on $\RP^3$:
\begin{corollary} \label{cor:homog-metrics-RP3}
    If $g$ is a homogeneous metric on $\RP^3$, then $g$ is isometric to the quotient of a homogeneous metric $\tilde{g}$ on $SU(2)$ by the map $q \mapsto -q$. We have $\I^0(\RP^3, g) = \I^0(SU(2), \tilde{g})/\Z_2$.
\end{corollary}
\begin{proof}
    Let $g$ be as in the corollary. Then the universal cover $(SU(2), \tilde{g})$ is homogeneous, and there is an embedding $\iota: \Z_2 \hookrightarrow \I(SU(2), \tilde{g})$ such that $\iota(\Z_2)$ acts freely and isometrically on $(SU(2), \tilde{g})$ with quotient isometric to $(\RP^3, g)$. Note that in each case (a)--(c) of Lemma \ref{lem:homog-metrics-S3}, the center of $\I^0(SU(2), \tilde{g})$ contains a unique copy of $\Z_2$ generated by the antipodal map $q \mapsto -q$. To prove the first part of the corollary, we need to show that $\iota(\Z_2)$ is exactly this central copy of $\Z_2$. We do this for cases (b) and (c); case (a) is more direct. The second part of the corollary follows from Lemma \ref{lem:quotient-isom}.

    (a) If $\tilde{g}$ is round, then so is $g$. Since diffeomorphic spherical space forms are isometric \cite{derham}, and a round $\RP^3$ can be obtained by quotienting $(SU(2),\tilde{g})$ by the map $q \mapsto -q$, it follows that $(\RP^3, g)$ is isometric to this quotient.

    (b) If $\tilde{g}$ is a Berger metric, then $\I(SU(2), \tilde{g}) \cong U(2) \rtimes \Z_2$. By Lemma \ref{lem:CN-equality} and Lemma \ref{lem:quotient-isom}, we have
    \begin{align} \label{eq:206}
        \dim C_{\I(SU(2), \tilde{g})}^0(\iota(\Z_2)) = \dim N_{\I(SU(2), \tilde{g})}^0(\iota(\Z_2)) = \dim \I^0(\RP^3, g) \geq 3,
    \end{align}
    where the last inequality holds because $g$ is homogeneous.
    We have $C_{\I(SU(2), \tilde{g})}^0(\iota(\Z_2)) \subset \I^0(SU(2),\tilde{g}) \cong U(2)$, so \eqref{eq:206} implies that $C_{\I(SU(2), \tilde{g})}^0(\iota(\Z_2))$ is either $SU(2)$ or $U(2)$. Either way, it follows that $\iota(\Z_2) \subset C_{\I(SU(2), \tilde{g})}(SU(2))$. This centralizer is computed to consist of right-multiplications by $e^{\mathbf{i}\theta}$ together with right-multiplications by $\mathbf{j}e^{\mathbf{i}\theta}$. Thus, $C_{\I(SU(2), \tilde{g})}(SU(2)) \cong \{ e^{\mathbf{i}\theta} \} \cup \mathbf{j}\{ e^{\mathbf{i}\theta} \} \cong \mathrm{Pin}(2)$.
    Since the unique order 2 subgroup of $\mathrm{Pin}(2)$ is contained in its identity component, it follows that $\iota(\Z_2) \subset Z(\I^0(SU(2),\tilde{g}))$, as required.

    (c) If $\tilde{g}$ is a generic left-invariant metric, then $\I(SU(2), \tilde{g}) \cong SU(2) \rtimes (\Z_2)^2$. The inequality \eqref{eq:206} still holds, but now it implies $C_{\I(SU(2), \tilde{g})}^0(\iota(\Z_2)) = SU(2)$. Thus,
    \begin{align}
        \iota(\Z_2) \subset C_{\I(SU(2),\tilde{g})}(SU(2)) = Z(SU(2)) = Z(\I^0(SU(2),\tilde{g})),
    \end{align}
    where the first equality is easily checked using Case (c) of Lemma \ref{lem:homog-metrics-S3}.
\end{proof}

The next lemma follows from Lemma \ref{lem:homog-metrics-S3} together with standard facts about subgroups of $SO(4)$ and $U(2)$.
\begin{lemma} \label{lem:subalg-classification-s3}
    Let $\tilde{g}$ be a homogeneous metric on $\Sph^3$. For each case below, we list all Lie subalgebras $\mathfrak{g}$ of $\mathfrak{I}(\Sph^3, \tilde{g})$ with $\dim\mathfrak{g} \geq 3$, up to conjugacy in $\I^0(\Sph^3,\tilde{g})$. For each $\mathfrak{g}$, we identify the unique connected Lie subgroup $G$ of $\I^0(\Sph^3, \tilde{g})$ with Lie algebra $\mathfrak{g}$. This gives a complete list of connected Lie subgroups $G \subseteq \I^0(\Sph^3, \tilde{g})$ with $\dim G \geq 3$, up to conjugacy. Finally, we determine whether $G$ acts transitively on $\Sph^3$; if so, we identify its isotropy subgroup $H \subseteq G$.

    \begin{enumerate}[label=(\alph*)]
        \item If $\tilde{g}$ is a round metric, i.e. $\I^0(\Sph^3, \tilde{g}) = SO(4) \cong SU(2)_L \times_{\Z_2} SU(2)_R$ and $\mathfrak{I}(\Sph^3,\tilde{g}) = \mathfrak{so}(4) \cong \mathfrak{su}(2)_L \oplus \mathfrak{su}(2)_R$, then the possibilities are shown in Table \ref{table:2}.

        \begin{table}[h]
            \centering
            \begin{threeparttable}
            \begin{tabular}{|c|c|c|}
            \hline
            $\mathfrak{g}$ & $G$ & $H$ (if $G$ transitive) \\
            \hline
            $\mathfrak{so}(4)$ & $SO(4)$ & $SO(3)$ \\
            $\mathfrak{su}(2)_L \oplus \mathfrak{u}(1)_R$ & $SU(2)_L \times_{\Z_2} U(1)_R \cong U(2)$ & $\mathrm{diag}(U(1))/\Z_2 \cong U(1)$ \\
            $\mathfrak{u}(1)_L \oplus \mathfrak{su}(2)_R$ & $U(1)_L \times_{\Z_2} SU(2)_R \cong U(2)$ & $\mathrm{diag}(U(1))/\Z_2 \cong U(1)$ \\
            $\mathfrak{su}(2)_L$ & $SU(2)_L$ & $\{1\}$ \\
            $\mathfrak{su}(2)_R$ & $SU(2)_R$ & $\{1\}$ \\
            $\mathrm{diag}(\mathfrak{su}(2))$ & $\mathrm{diag}(SU(2))/\Z_2 \cong SO(3)$ &
            $G$ not transitive\tnote{*} \\
            \hline
            \end{tabular}
            
            \begin{tablenotes}
            \footnotesize
            \item[*] In this case, $G$ acts by rotations of $\Sph^3 \subset \R^4$ that fix an axis.
            \end{tablenotes}
            \end{threeparttable}
            \caption{Possibilities for $\mathfrak{g}$, $G$, and $H$ if $\tilde{g}$ is a round metric.}
            \label{table:2}
        \end{table}
        \item If $\tilde{g}$ is a Berger metric, i.e. $\I^0(\Sph^3, \tilde{g}) = U(2)$ and $\mathfrak{I}(\Sph^3, \tilde{g}) = \mathfrak{u}(2)$, then the possibilities are shown in Table \ref{table:1}.
        \begin{table}[h]
            \centering
            \begin{tabular}{|c|c|c|}
                \hline
                $\mathfrak{g}$ & $G$ & $H$ (if $G$ transitive) \\
                \hline
                $\mathfrak{u}(2)$ & $U(2)$ & $U(1)$ \\
                $\mathfrak{su}(2)$ & $SU(2)$ & $\{1\}$ \\
                \hline
            \end{tabular}
            \caption{Possibilities for $\mathfrak{g}$, $G$, and $H$ if $\tilde{g}$ is a Berger metric.}
            \label{table:1}
        \end{table}
        \item If $\tilde{g}$ is a generic left-invariant metric, i.e. $\I^0(\Sph^3,\tilde{g}) = SU(2)$ and $\mathfrak{I}(\Sph^3, \tilde{g}) = \mathfrak{su}(2)$, then the only possibility is $\mathfrak{g} = \mathfrak{su}(2)$ and $G = SU(2)$. In this case, $G$ acts transitively on $\Sph^3$ with trivial isotropy subgroup $H$.
    \end{enumerate}
\end{lemma}

\begin{corollary} \label{cor:subalg-classification-rp3}
    Using Corollary \ref{cor:homog-metrics-RP3}, we immediately obtain an analog of Lemma \ref{lem:subalg-classification-s3} for homogeneous metrics $g$ on $\RP^3$:
    \begin{enumerate}[label=(\alph*)]
        \item In each case for $g$ (round, Berger, generic left-invariant), we can classify all Lie subalgebras $\mathfrak{g}$ of $\mathfrak{I}(\RP^3, g)$ with $\dim\mathfrak{g} \geq 3$ up to conjugacy in $\I^0(\RP^3,g)$. The list is exactly the same as that given in Lemma \ref{lem:subalg-classification-s3}.
        \item For each $\mathfrak{g}$, we can identify the unique connected Lie subgroup $G \subset \I^0(\RP^3, g)$. These are simply the $G$'s from Lemma \ref{lem:subalg-classification-s3} quotiented by central $\Z_2$'s. This gives a complete list of connected Lie subgroups $G \subseteq \I^0(\RP^3, g)$ with $\dim G \geq 3$, up to conjugacy.
        \item All $G$'s act transitively on $\RP^3$, except for $G$ coming from the last row of Table \ref{table:2}. For the transitive $G$'s, the isotropy subgroup $H$ is the image of the corresponding isotropy subgroup from Lemma \ref{lem:subalg-classification-s3} under the quotient map $G \to G/\Z_2$ from the previous point.
    \end{enumerate}
\end{corollary}

\begin{corollary} \label{cor:RP3-subgroups-transitive}
    Let $g$ be a homogeneous metric on $\RP^3$ with $\dim \I(\RP^3, g) \geq 4$. Then any Lie subgroup $G \subseteq \I^0(\RP^3, g)$ with $\dim G \geq 4$ must act transitively.
\end{corollary}
\begin{proof}
    By Corollary \ref{cor:homog-metrics-RP3}, $g$ must be the standard $\Z_2$-quotient of a homogeneous metric $\tilde{g}$ on $\Sph^3$. Since $\dim \I(\RP^3, g) \geq 4$, we have $\dim \I(\Sph^3, \tilde{g}) \geq 4$ as well. Thus, by Lemma \ref{lem:homog-metrics-S3}, $\tilde{g}$ is either a round metric or a Berger metric.
    \begin{itemize}
        \item If $\tilde{g}$ is round, then $\I^0(\RP^3, g) = SO(4)/\Z_2$. By Corollary \ref{cor:subalg-classification-rp3}, any Lie subalgebra of $\mathfrak{I}(\RP^3, g) = \mathfrak{so}(4)$ with dimension at least 4 must be $\mathfrak{so}(4)$, $\mathfrak{su}(2)_L \oplus \mathfrak{u}(1)_R$, or $\mathfrak{u}(1)_L \oplus \mathfrak{su}(2)_R$. In the first case, the Killing fields clearly span the tangent space everywhere, so $G$ acts transitively. In the second (resp. third) case, the corresponding subalgebra of Killing fields contains those generating left (resp. right) multiplication in $SU(2)$, so they also span the tangent space at every point.
        \item If $\tilde{g}$ is a Berger metric, then $\I^0(\RP^3, g) = U(2)/\Z_2$, which has dimension 4. Thus, $G = \I^0(\RP^3, g)$, which acts transitively.
    \end{itemize}
\end{proof}

\begin{corollary} \label{cor:closed-subgroup-rp3}
    If $g$ is a metric on $\RP^3$ with $\dim \I(\RP^3, g) \geq 3$, then every connected Lie subgroup $G$ of $\I^0(\RP^3, g)$ with dimension $\geq 3$ is compact.
\end{corollary}
\begin{proof}
    If $\dim \I(\RP^3,g) = 3$ this is trivial, so we assume $\dim \I(\RP^3, g) \geq 4$. It is well-known that this implies $g$ is homogeneous \cite{kobayashi-groups}*{p. 48--49}. Corollary \ref{cor:subalg-classification-rp3}(b) classifies all connected Lie subgroups $G$ of $\I^0(\RP^3, g)$ with $\dim G \geq 3$. It is clear that each $G$ is compact.
\end{proof}

\subsection{Proof of Theorem \ref{thm:EH-characterization}: dividing into cases} \label{subsec:case-split}

For the rest of \S\ref{sec:eh-result}, $(M^4,g)$ is a complete 4-dimensional Ricci-flat ALE space with finitely many orbifold points and group at infinity $\Z_2$. Moreover, we assume that $\dim \ker_{L^2}(\Delta_L) \leq 3$.
We now proceed to prove Theorem \ref{thm:EH-characterization}.

By Proposition \ref{prop:exact-KFs} and the setup before Corollary \ref{cor:g-jacobi}, there is a Lie subalgebra $\mathfrak{g}$ of Killing fields on $(M,g)$ with $\dim\mathfrak{g} \geq 6-3=3$, which for large $\rho$ restricts to a Lie algebra $\mathfrak{g}_\rho$ of Killing fields on the CMC leaves $(\RP^3, g_\rho)$. These Lie algebras integrate to connected Lie groups $G \subseteq \I^0(M,g)$ and $G_\rho \subseteq \I^0(\RP^3, g_\rho)$ of dimension $\geq 3$.

By Corollary \ref{cor:closed-subgroup-rp3}, $G_\rho$ is compact. By Lemma \ref{lem:level-set-isos}, $G$ acts smoothly, isometrically, and effectively on $(M,g)$, and the restriction map $G \to G_\rho$ is an isomorphism. In particular, $G_\rho$ for all large $\rho$ (say $\rho \geq R_0$) are isomorphic. Consider two cases:
\begin{enumerate}[label=\textbf{(\Alph*)}]
    \item $G$ acts transitively on $(\RP^3, g_\rho)$ for some $\rho \geq R_0$.
    \item $G$ does not act transitively on $(\RP^3, g_\rho)$ for all $\rho \geq R_0$.
\end{enumerate}
In \S\ref{subsec:caseA}, we will show that Case (A) implies $(M,g)$ is either the Eguchi--Hanson space or the flat orbifold $\R^4/\Z_2$. In \S\ref{subsec:caseB}, we will show that Case (B) necessarily leads to $(M,g)$ being the flat $\R^4/\Z_2$. Altogether, this will prove Theorem \ref{thm:EH-characterization}.

\subsection{Case (A)} \label{subsec:caseA}

In this subsection, we assume $G$ acts transitively on $(\RP^3, g_\rho)$, so $g_\rho$ is homogeneous. We will show that $(M,g)$ is either Eguchi--Hanson space or the flat $\R^4/\Z_2$.

Corollary \ref{cor:homog-metrics-RP3} determines all possibilities for $g_\rho$ and $\I^0(\RP^3, g_\rho)$. Based on this, Corollary \ref{cor:subalg-classification-rp3} allows us to classify $G$ (up to isomorphism) and the isotropy subgroup $H$. Table \ref{table:3} lists all possibilities.
\begin{table}[h]
    \centering
    \begin{tabular}{|c|c|c|c|}
        \hline
        $g_\rho$ & $\I^0(\RP^3, g_\rho)$ & $G$ & $H$ \\
        \hline
        \multirow{3}{*}{Round} &
        \multirow{3}{*}{$SO(4)/\Z_2$} &
        $SO(4)/\Z_2$ & $SO(3)$ \\
        \cline{3-4}
         & & $U(2)/\Z_2$ & $U(1)$ \\
        \cline{3-4}
         & & $SU(2)/\Z_2$ & $\{1\}$ \\
        \hline
        \multirow{2}{*}{Berger} &
        \multirow{2}{*}{$U(2)/\Z_2$} &
        $U(2)/\Z_2$ & $U(1)$ \\
        \cline{3-4}
         & & $SU(2)/\Z_2$ & $\{1\}$ \\
        \hline
        Generic left-invariant & $SU(2)/\Z_2$ & $SU(2)/\Z_2$ & $\{1\}$ \\
        \hline
    \end{tabular}
    \caption{All possibilities for $g_\rho$, $\I^0(\RP^3, g_\rho)$, $G$, $H$ in Case (A).}
    \label{table:3}
\end{table}

By assumption, there exists a pair $(G,H)$ from Table \ref{table:3} such that $G$ acts smoothly, effectively, and isometrically on $(M,g)$ with cohomogeneity one and principal isotropy $H$. To simplify things, we will modify $G$ and $H$ by using their preimages under the $\Z_2$-quotients. Thus, for some pair $(G,H)$ in Table \ref{table:4}, $G$ acts smoothly, \emph{almost} effectively, and isometrically with cohomogeneity one on $(M,g)$ and principal isotropy $H$.
\begin{table}[h]
    \centering
    \begin{tabular}{|c|c|}
        \hline
        $G$ & $H$ \\
        \hline
        $SO(4)$ & $O(3)$ \\
        $U(2)$ & $U(1) \times \Z_2$ \\
        $SU(2)$ & $\Z_2$ \\
        \hline
    \end{tabular}
    \caption{In Case (A), some $G$ here acts smoothly, almost effectively, and isometrically with cohomogeneity one on $(M,g)$ and principal isotropy $H$.}
    \label{table:4}
\end{table}

By Proposition \ref{prop:coho-1-structure}, it holds that for some $(G,H)$ from Table \ref{table:4},
\begin{enumerate}[label=(\roman*)]
    \item Outside a submanifold $X$, we can identify
    \begin{align}
        (M \setminus X, g) = \left( (0,\infty) \times G/H, ds^2 + g(s) \right)
    \end{align}
    where $\{g(s)\}_{s>0}$ are $G$-invariant metrics on a fixed copy of $G/H \cong \RP^3$.
    \item There is a closed subgroup $K$ of $G$ strictly containing $H$ such that $X$ is diffeomorphic to $G/K$, and $K/H$ is diffeomorphic to a spherical space form of dimension $\ell \geq 0$.
    \item From the description of $(M,g)$ as a bundle over $X$, the metric spaces $(G/H, g(s))$ converge in the Hausdorff sense to $X$ as $s \to 0$.
\end{enumerate}

We will use (ii) and (iii) with the following lemma, which follows from standard facts about closed subgroups of $SO(4)$, $U(2)$, and $SU(2)$. Below, $L(4,1)$ is the lens space obtained by quotienting $\Sph^3$ by a Hopf action of $\Z_4$.
\begin{lemma} \label{lem:G/K-possibilities}
    For each $(G,H)$ in Table \ref{table:4}, we indicate in Table \ref{table:5} all closed subgroups $K$ of $G$ such that (i) $K$ strictly contains $H$, and (ii) $K/H$ is diffeomorphic to a spherical space form of dimension $\ell \geq 0$. We also indicate the diffeomorphism type of $G/K$.
    \begin{table}[h]
        \centering
        \begin{tabular}{|c|c|c|c|}
            \hline
            $G$ & $H$ & $K$ & $G/K$ \\
            \hline
            $SO(4)$ & $O(3)$ & $SO(4)$ & $\{\text{point}\}$ \\
            \hline
            \multirow{3}{*}{$U(2)$} &
            \multirow{3}{*}{$U(1)\times \Z_2$} &
            $U(1)\times \Z_4$ & $L(4,1)$ \\
            \cline{3-4}
             & & $U(1)\times U(1)$ & $\Sph^2$ \\
            \cline{3-4}
            & & $U(2)$ & $\{\text{point}\}$ \\
            \hline
            \multirow{3}{*}{$SU(2)$} &
            \multirow{3}{*}{$\Z_2$} &
            $\Z_4$ & $L(4,1)$ \\
            \cline{3-4}
             & & $U(1)$ & $\Sph^2$ \\
            \cline{3-4}
             & & $SU(2)$ & $\{\text{point}\}$ \\
            \hline
        \end{tabular}
        \caption{Results for Lemma \ref{lem:G/K-possibilities}.}
        \label{table:5}
    \end{table}
\end{lemma}

Since the metrics $g(s)$ are homogeneous, Corollary \ref{cor:homog-metrics-RP3} implies that $g(s)$ is the $\Z_2$-quotient of a metric of the form
\begin{align} \label{eq:bianchiIX}
    \tilde{g}(s) = ds^2 + \lambda(s)^2 \sigma_1(s) \otimes \sigma_1(s) + \mu(s)^2 \sigma_2(s) \otimes \sigma_2(s) + \nu(s)^2 \sigma_3(s) \otimes \sigma_3(s)
\end{align}
on $(0,\infty) \times SU(2)$, where $\lambda,\mu,\nu$ are smooth positive functions, and for each $s > 0$, $\{\sigma_1(s),\sigma_2(s),\sigma_3(s)\}$ is a left-invariant coframe on $SU(2)$. It is well-known that Ricci-flatness leads to diagonalizability of such metrics, i.e. $\sigma_i(s)$ can be chosen independently of $s$ (see e.g. \cite{dammermann}*{\S 3} for a proof). Thus,
\begin{align}
    g(s) = \lambda(s)^2 \sigma_1 \otimes \sigma_1 + \mu(s)^2 \sigma_2 \otimes \sigma_2 + \nu(s)^2 \sigma_3 \otimes \sigma_3.
\end{align}
From this, we see that any Hausdorff limit of $(\RP^3, g(s))$ as $s \to 0$ is homeomorphic to either $\RP^3$ or a space of lower dimension. This rules out $G/K \cong L(4,1)$, so Lemma \ref{lem:G/K-possibilities} leaves us with two possibilities: either $G/K \cong \Sph^2$ or $G/K$ is a point.
\begin{itemize}
    \item If $G/K$ is a point, then $M$ is diffeomorphic to the orbifold $\R^4/\Z_2$. Its universal orbifold cover is $(\R^4, \tilde{g})$ where $\tilde{g}$ is a Ricci-flat, asymptotically Euclidean metric. By the sharp case of the Bishop--Gromov volume comparison theorem, $\tilde{g}$ is the Euclidean metric. Hence $g$ is flat $\R^4/\Z_2$.
    \item If $G/K = \Sph^2$, then $M$ is a smooth manifold. Indeed, either all points in the orbit $G/K$ are orbifold points, or they are all smooth points. Since $M$ is assumed to have finitely many orbifold points, the latter must be true. As the principal part $(0,\infty) \times G/H$ is already smooth, it follows that $M$ is smooth. By construction, $M$ is diffeomorphic to $\O_{\CP^1}(-2) \cong T^*\Sph^2$. By a theorem of Lock and Viaclovsky \cite{lock-viaclovsky}*{Theorem 1.5}, $(M,g)$ is hyperk\"ahler. Kronheimer's classification \cite{kronheimer} now implies that $(M,g)$ is Eguchi--Hanson space.
\end{itemize}
This concludes Case (A).

\subsection{Case (B)} \label{subsec:caseB}

In this subsection, we assume $G$ does not act transitively on $(\RP^3,g_\rho)$ for all $\rho \geq R_0$. We will show that $(M,g)$ must be the flat $\R^4/\Z_2$.
\begin{corollary} \label{cor:Gdim}
    We have $\dim G = 3$.
\end{corollary}
\begin{proof}
    From \S\ref{subsec:case-split}, we already know that $\dim G \geq 3$. If $\dim G \geq 4$, then since $G \cong G_\rho \subseteq \I^0(\RP^3, g_\rho)$ for each $\rho \geq R_0$, it follows that $\dim \I(\RP^3, g_\rho) \geq 4$. This makes $g_\rho$ homogeneous \cite{kobayashi-groups}*{p. 48--49}.
    By Corollary \ref{cor:RP3-subgroups-transitive}, $G$ must act transitively on $(\RP^3,g_\rho)$, a contradiction.
\end{proof}

Let $(\tilde{M}_\infty, \tilde{g})$ be the Riemannian universal cover of the end of $(M,g)$. Then $(\tilde{M}_\infty, \tilde{g})$ is an incomplete asymptotically Euclidean manifold of order 4 diffeomorphic to $\R^4 \setminus B_R(0)$, and it is foliated by CMC 3-spheres. Lifting the action of $G$ on the end of $(M,g)$, we obtain a compact covering group $\tilde{G}$ acting isometrically and effectively on $(\tilde{M}_\infty, \tilde{g})$. By Corollary \ref{cor:Gdim}, $\dim\tilde{G} = 3$. This restricts to an isometric and effective action on all CMC leaves $(\Sph^3, \tilde{g}_\rho)$, so $\tilde{G}$ can be identified as a subgroup of $\I(\Sph^3, \tilde{g}_\rho)$. Its action on $(\Sph^3, \tilde{g}_\rho)$ is not transitive. Let $\tilde{G}^0$ be its identity component. By the classification of compact connected Lie groups of dimension 3, $\tilde{G}^0$ is $SU(2)$, $SO(3)$, or $\mathbb{T}^3$. However, the next lemma excludes $\tilde{G}^0 = \mathbb{T}^3$. Although this is well-known, we supply a proof for the reader's convenience.

\begin{lemma}
    $\mathbb{T}^3$ cannot act smoothly and effectively on $\Sph^3$.
\end{lemma}
\begin{proof}
    Suppose such an action exists. Since $\mathbb{T}^3$ is abelian, the isotropy subgroup $H \subset \mathbb{T}^3$ of all points in principal orbits are equal (not just conjugate). Thus, if $H$ is nontrivial, then there exists a nontrivial element $\phi \in H$ which fixes all points in principal orbits. Since principal orbits form a dense open set, it follows that $\phi$ fixes all points on $\Sph^3$. However, this contradicts effectiveness, so $H$ must be trivial. It follows that each principal orbit of the action is 3-dimensional (diffeomorphic to $\mathbb{T}^3$), hence open in $\Sph^3$. On the other hand, all orbits of the action must be closed. Therefore, a principal orbit must be the whole of $\Sph^3$. But this implies $\Sph^3$ is diffeomorphic to $\mathbb{T}^3$, a contradiction.
\end{proof}

Next we determine the cohomogeneity of the action of $\tilde{G}^0$ on $(\tilde{M}_\infty, \tilde{g})$, which is one plus the cohomogeneity of the action restricted to $(\Sph^3, \tilde{g}_\rho)$. On $(\Sph^3, \tilde{g}_\rho)$, the principal orbits have positive dimension as the action is effective, but they cannot have dimension 3 as the action is not transitive. Neither can they have dimension 1, because $SO(3)$ and $SU(2)$ do not contain closed subgroups of dimension 2. Therefore, the principal orbits have dimension 2, so the action has cohomogeneity one on $(\Sph^3, \tilde{g}_\rho)$. This also rules out $\tilde{G}^0 = SU(2)$, as $SU(2)$ cannot act effectively with cohomogeneity one on a 3-sphere (see e.g. \cite{hsiang-lawson}).
Thus, $\tilde{G}^0 = SO(3)$. The next proposition summarizes the preceding discussion.

\begin{proposition} \label{prop:action-classification}
    $\tilde{G}^0 \cong SO(3)$ acts smoothly, isometrically, and effectively on $(\tilde{M}_\infty, \tilde{g})$ with cohomogeneity two. The action preserves all CMC leaves $(\Sph^3, \tilde{g}_\rho)$, and restricts to act with cohomogeneity one on each leaf.
\end{proposition}

\begin{corollary} \label{cor:structure-coho2}
    \begin{enumerate}[label=(\alph*)]
        \item The action has principal isotropy subgroup $H \cong SO(2)$.
        \item Each principal orbit is diffeomorphic to $\Sph^2$ and the metric on it is round.
        \item The induced metric on each CMC leaf $(\Sph^3, \tilde{g}_\rho)$ is of the form
        \begin{align}
            ds^2 + f(s)^2 g_{\Sph^2},
        \end{align}
        where $s \in (0,1)$, and $f: (0,1) \to \R$ is a smooth positive function with $f(0) = f(1) = 0$.
    \end{enumerate}
\end{corollary}
\begin{proof}
    Since the principal orbits are 2-dimensional, the principal isotropy subgroup $H$ is a 1-dimensional closed Lie subgroup of $SO(3)$. The only ones up to conjugacy are $SO(2)$ and $O(2)$. If $H \cong O(2)$, then each principal orbit is an $SO(3)/O(2) \cong \RP^2$ embedded in a CMC leaf of $\tilde{M}_\infty$. But it is impossible for $\RP^2$ to embed into $\Sph^3$. Hence, $H \cong SO(2)$ and part (a) follows. The principal orbits are then $SO(3)/SO(2) \cong \Sph^2$ equipped with $SO(3)$-invariant metrics, which are necessarily round. This proves (b). Finally, (c) follows by restricting to a CMC leaf where the action has cohomogeneity one, then applying the structure theory of cohomogeneity one metrics (e.g. \cite{alexandrino-bettiol}*{\S 6.3}). (This is similar to Proposition \ref{prop:coho-1-structure}, except the orbit space is a closed interval in this case.)
\end{proof}

Let $\tilde{M}_\infty^{\mathrm{prin}} \subset \tilde{M}_\infty$ be the union of principal orbits with respect to the action of $\tilde{G}^0$. The orbit space $B := \tilde{M}_\infty^{\mathrm{prin}}/\tilde{G}^0$ is a connected Riemannian manifold \cite{alexandrino-bettiol}*{Theorem 3.82}.

We claim that $B$ is diffeomorphic to $\R^2$. Indeed, Corollary \ref{cor:structure-coho2} shows that each CMC leaf carries a cohomogeneity one $\tilde{G}^0$-action whose principal part quotients to an open interval $(0,1)$, while $\tilde{M}_\infty$ itself is foliated by these CMC leaves and is diffeomorphic to $(R_0,\infty) \times \Sph^3$. Thus, the orbit space $B$ of $\tilde{M}_\infty^{\mathrm{prin}}$ is $(R_0,\infty) \times (0,1)$, which is diffeomorphic to $\R^2$.

For each $x \in \tilde{M}_\infty^{\mathrm{prin}}$, we have
\begin{align}
    T_x \tilde{M}_\infty^{\mathrm{prin}} = \mathcal{H}_x \oplus \mathcal{V}_x,
\end{align}
where $\mathcal{V}_x$ is the tangent space to the $\Sph^2$ orbit and $\mathcal{H}_x = \mathcal{V}_x^\perp$ is its $\tilde{g}$-orthogonal complement. We call $\mathcal{H}_x$ and $\mathcal{V}_x$ the \emph{horizontal} and \emph{vertical} subspaces at $x$ respectively.
Since $\tilde{G}^0$ acts isometrically, it preserves the subbundles $\mathcal{H}$ and $\mathcal{V}$. Restricting to the isotropy subgroup $H$ at $x$ gives the following.
\begin{itemize}
    \item $\mathcal{H}_x$ carries the slice representation of $H = SO(2)$. This is a trivial representation because the orbit of $x$ is principal \cite{alexandrino-bettiol}*{\S 3.4}.
    \item $\mathcal{V}_x$ carries the isotropy representation of $H$, which is identified with the adjoint representation of $H$ on an $\mathrm{Ad}(H)$-invariant complement of $\mathfrak{h}$ in $\mathfrak{g}$. As one checks, this is isomorphic to the defining representation of $SO(2)$.
\end{itemize}
Let $(\cdot)^{\mathcal{H}}$ and $(\cdot)^{\mathcal{V}}$ denote orthogonal projections of tangent vectors to $\mathcal{H}$ and $\mathcal{V}$ respectively.
Following O'Neill \cite{oneill}, we define tensor fields $A$ and $T$ by setting at each $x \in \tilde{M}_\infty^{\mathrm{prin}}$
\begin{align}
    A_x: \wedge^2\mathcal{H}_x \to \mathcal{V}_x, \quad &A_x(X,Y) = \frac{1}{2}[\tilde{X},\tilde{Y}]^{\mathcal{V}}, \\
    T_x: \mathcal{V}_x \otimes \mathcal{V}_x \to \mathcal{H}_x, \quad &T_x(Z,W) = (\nabla_{\tilde{Z}} \tilde{W})^{\mathcal{H}},
\end{align}
where $\tilde{X}$ and $\tilde{Y}$ are local horizontal extensions of $X$ and $Y$, and $\tilde{Z}$ and $\tilde{W}$ are local vertical extensions of $Z$ and $W$. Vanishing of $A$ is equivalent to integrability of the horizontal distribution. Meanwhile, $T$ is the second fundamental form of the $\Sph^2$ orbits. It is straightforward to verify that for all $\phi \in \tilde{G}^0$,
\begin{align} \label{eq:AT-G-equivariant}
    A(d\phi(X), d\phi(Y)) = d\phi(A(X,Y)), \quad T(d\phi(Z), d\phi(W)) = d\phi(T(Z,W)).
\end{align}
In particular, $A_x$ and $T_x$ are $H$-equivariant maps at each $x \in \tilde{M}_\infty^{\mathrm{prin}}$.

\begin{lemma} \label{lem:submersion-properties}
    \begin{enumerate}[label=(\alph*)]
        \item $A \equiv 0$, i.e. the horizontal distribution $\mathcal{H} \subset T\tilde{M}_\infty^{\mathrm{prin}}$ is integrable.
        \item Each principal orbit is \emph{umbilic}, i.e. $T_x(Z,W) = \tilde{g}_x(Z,W) \eta_x$ for all $Z,W \in \mathcal{V}_x$ where $\eta_x \in \mathcal{H}_x$ is the mean curvature vector.
        \item Each principal orbit is \emph{spheric}, i.e. the mean curvature vector field $\eta$ along any principal orbit satisfies $(\nabla_Z \eta)^{\mathcal{H}} = 0$ for all vertical vectors $Z \in \mathcal{V}_x$.
    \end{enumerate}
\end{lemma}
\begin{proof}
    (a) Since $\mathcal{H}_x$ is a trivial $H$-module, so is $\wedge^2\mathcal{H}_x$. Also, $\mathcal{V}_x$ has no nonzero vectors fixed by $H$. Since $A_x$ is $H$-equivariant, these facts force $A_x = 0$.

    (b) Since $T_x$ is $H$-equivariant, and $\mathcal{H}_x$ is a trivial $H$-module, it follows that $T_x$ is $H$-invariant. Take a basis $\{e_1,e_2\}$ for $\mathcal{H}_x$ and write $T_x(Z,W) = T_x^{(1)}(Z,W) e_1 + T_x^{(2)}(Z,W) e_2$; then $T_x^{(i)}: \mathcal{V}_x \otimes \mathcal{V}_x \to \R$ is $H$-invariant for $i=1,2$. By Schur's lemma, the space of $H$-invariant symmetric 2-tensors on $\mathcal{V}_x$ is one-dimensional, spanned by $\tilde{g}_x$. Hence, $T_x^{(1)}$ and $T_x^{(2)}$ each are multiples of $\tilde{g}_x$, varying smoothly in $x$. Thus, $T_x(Z,W) = \tilde{g}_x(Z,W) \nu_x$ for some $\nu_x \in \mathcal{H}_x$. As $T_x$ is the second fundamental form of the orbit at $x$, we must have $\nu_x = \eta_x$.

    (c) Using part (b) and \eqref{eq:AT-G-equivariant}, we see that $d\phi(\eta_x) = \eta_{\phi(x)}$ for each $\phi \in \tilde{G}^0$. Using this, one verifies that
    \begin{align} \label{eq:59682}
        d\phi((\nabla_Z \eta)^{\mathcal{H}}) = (\nabla_{d\phi(Z)}\eta)^{\mathcal{H}} \quad \text{for all } \phi \in \tilde{G}^0, \, Z \in \mathcal{V}_x.
    \end{align}
    Restricting to $\phi \in H$, \eqref{eq:59682} implies that $(\nabla\eta)^{\mathcal{H}}: \mathcal{V}_x \to \mathcal{H}_x$ is $H$-equivariant. Again, by the triviality of $\mathcal{H}_x$ and the fact that $\mathcal{V}_x$ has no points fixed by $H$, this yields $(\nabla\eta)^{\mathcal{H}} = 0$ at $x$.
\end{proof}

Recall that $B = \tilde{M}_\infty^{\mathrm{prin}}/\tilde{G}^0$ is a connected Riemannian manifold diffeomorphic to $\R^2$.
\begin{proposition} \label{prop:prin-warped}
    $(\tilde{M}_\infty^{\mathrm{prin}}, \tilde{g})$ is isometric to a warped product $(B \times \Sph^2, h + \xi^2 g_{\Sph^2})$, where $h$ is a metric on $B$ and $\xi: B \to (0,\infty)$ is a smooth positive function.
\end{proposition}
\begin{proof}
    By \cite{alexandrino-bettiol}*{Remark 3.99}, the quotient map $\pi: \tilde{M}_\infty^{\mathrm{prin}} \to B$ is a Riemannian submersion. By Corollary \ref{cor:structure-coho2}, the fibers are round 2-spheres. Fix a basepoint $b_0 \in B$ and identify $\pi^{-1}(b_0)$ with $\Sph^2$. For any $b \in B$, any smooth path $\gamma: [0,1] \to B$ with $\gamma(0) = b_0$ and $\gamma(1) = b$, and any $\theta \in \Sph^2$, let $\tilde{\gamma}: [0,1] \to \tilde{M}_\infty^{\mathrm{prin}}$ be the horizontal lift of $\gamma$ starting at $\theta$. Horizontal transport along $\tilde{\gamma}$ defines a diffeomorphism $P_\gamma: \Sph^2 \cong \pi^{-1}(b_0) \to \pi^{-1}(b)$. Since the horizontal distribution $\mathcal{H}$ is integrable by Lemma \ref{lem:submersion-properties}(a), $P_\gamma$ depends only on the homotopy class of $\gamma$. But $B \cong \R^2$ is contractible, so $P_\gamma$ is independent of the path $\gamma$ joining $b_0$ to $b$. Thus, there is a well-defined smooth map $P_b: \Sph^2 \to \pi^{-1}(b)$. A global diffeomorphism is given by
    \begin{align}
        B \times \Sph^2 \to \tilde{M}_\infty^{\mathrm{prin}}, \quad (b,\theta) \mapsto P_b(\theta).
    \end{align}
    Under this diffeomorphism, $\mathcal{H}$ is identified with $TB$ and $\mathcal{V}$ is identified with $T\Sph^2$. Since $A \equiv 0$, the submanifold $B \times \{\theta\} \subset B \times \Sph^2$ is totally geodesic for each $\theta \in \Sph^2$. Moreover, for each $b' \in B$, the submanifold $\{b'\} \times \Sph^2$ corresponds to a fiber of $\pi$, which by Lemma \ref{lem:submersion-properties} is spheric. We now apply \cite{ponge-reckziegel}*{Proposition 3(c)} to conclude that $(\tilde{M}_\infty^{\mathrm{prin}}, \tilde{g})$ is globally a warped product.
\end{proof}

The next lemma shows that under a suitable identification of $\tilde{M}_\infty$ with $\R^4 \setminus B_R(0)$, the action of $\tilde{G}^0$ asymptotes to Euclidean rotations fixing an axis. In Corollary \ref{cor:drho-not-0}, we draw some consequences of this.

\begin{lemma} \label{lem:limiting-action-Gtilde0}
    There is an asymptotically Euclidean diffeomorphism $\tilde{\Phi}: \tilde{M}_\infty \to \R^4 \setminus B_R(0)$ of order 4 and class $C^{100}$ (i.e. \eqref{eq:ALE-diffeo} holds only for $k \leq 100$) with the following properties.
    \begin{enumerate}[label=(\alph*)]
        \item $\tilde{\Psi}$ maps CMC 3-spheres in $\tilde{M}_\infty$ to coordinate spheres in $\R^4 \setminus B_R(0)$.
        \item Let $\tilde{\Psi}$ be the action of $\tilde{G}^0$ on $\tilde{M}_\infty \stackrel{\tilde{\Phi}}{\cong} \R^4 \setminus B_R(0)$, i.e.
        \begin{align}
            \tilde{\Psi}: \tilde{G}^0 \times (\R^4 \setminus B_R(0)) \to \R^4 \setminus B_R(0), \quad \tilde{\Psi}(\phi, x) = \tilde{\Phi}(\phi \cdot \tilde{\Phi}^{-1}(x)).
        \end{align}
        Also view $\Sph^3$ as the unit sphere in $\R^4$, and for $\rho > R$ define the rescaled actions
        \begin{align}
            \tilde{\Psi}_\rho: \tilde{G}^0 \times \Sph^3 \to \Sph^3, \quad \tilde{\Psi}_\rho(\phi,\theta) = \rho^{-1} \tilde{\Psi}(\phi, \rho\theta).
        \end{align}
        Then as $\rho \to \infty$, we have $\tilde{\Psi}_\rho \to \tilde{\Psi}_\infty$ in $C^{80}$, where
        \begin{align}
            \tilde{\Psi}_\infty: \tilde{G}^0 \times \Sph^3 \to \Sph^3.
        \end{align}
        is the standard $SO(3)$-action on $\Sph^3$ induced by Euclidean rotations of $\R^4$ fixing the $w$-axis, i.e. the span of $(0,0,0,1)$.
    \end{enumerate}
\end{lemma}
\begin{proof}
    By \cite{biquard-hein}*{Theorem B}, there is an ALE diffeomorphism $\Phi: M \setminus K \to (\R^4 \setminus B_R(0))/\Z_2$ of order 4 and class $C^{100}$ which maps CMC hypersurfaces to coordinate spheres (mod $\Z_2$).
    Lifting $\Phi$ gives an asymptotically Euclidean diffeomorphism $\tilde{\Phi}: \tilde{M}_\infty \to \R^4 \setminus B_R(0)$ of order 4 and class $C^{100}$ which maps CMC 3-spheres to coordinate spheres.
    
    Using $\tilde{\Phi}$, we define $\tilde{\Psi}$ and $\tilde{\Psi}_\rho$ as in the lemma. Let $\tilde{X}_1, \tilde{X}_2, \tilde{X}_3$ be fundamental Killing fields on $\R^4 \setminus B_R(0)$ generating the action $\tilde{\Psi}$ and satisfying the $\mathfrak{so}(3)$ bracket relations:
    \begin{align} \label{eq:06018}
        [\tilde{X}_1, \tilde{X}_2] = 2\tilde{X}_3, \quad [\tilde{X}_2, \tilde{X}_3] = 2\tilde{X}_1, \quad [\tilde{X}_3, \tilde{X}_1] = 2\tilde{X}_2.
    \end{align}
    Since $\{\tilde{X}_i\}_{i=1}^3$ are also fundamental Killing fields for the action of $\tilde{G}$ on $\tilde{M}_\infty$, which in turn lifts the action of $G$ on $M$, it follows that $\{\tilde{X}_i\}_{i=1}^3$ are lifts of fundamental Killing fields $X_1, X_2, X_3$ for the action of $G$ on $M$.
    
    Although $\Phi$ is not smooth, one checks that Proposition \ref{prop:exact-KFs} still holds, provided we only assert control up to say, 81 derivatives in parts (d) and (e). Indeed, parts (a)--(c) only require a small number of derivatives of ALE decay in \eqref{eq:ALE-diffeo}. Thus, Proposition \ref{prop:exact-KFs}(e) gives linearly independent Killing fields $\{X_{i,\infty}\}_{i=1}^3$ on $(\RP^3, \bar{g}_\infty)$ and $\epsilon \in (0,\frac{1}{2})$ such that the restrictions $X_{i,\rho} = X_i|_{\del B_\rho(0)}$ satisfy
    \begin{align}
        \lVert X_{i,\rho} - X_{i,\infty} \rVert_{C^{81}(\RP^3, \bar{g}_\infty)} \leq \O(\rho^{-3+\epsilon}).
    \end{align}
    Lifting, this gives linearly independent Killing fields $\{\tilde{X}_{i,\infty}\}_{i=1}^3$ on $(\Sph^3, \bar{g}_\infty)$ such that the restrictions $\tilde{X}_{i,\rho} = \tilde{X}_i|_{\del B_\rho(0)}$ satisfy
    \begin{align} \label{eq:Xtildei-convergence}
        \lVert \tilde{X}_{i,\rho} - \tilde{X}_{i,\infty} \rVert_{C^{81}(\Sph^3, \bar{g}_\infty)} \leq \O(\rho^{-3+\epsilon}).
    \end{align}
    In particular, the $\tilde{X}_{i,\infty}$ satisfy the same Lie algebra relations as \eqref{eq:06018}, so they generate a Lie subalgebra $\mathfrak{g}_\infty$ of $\mathfrak{I}(\Sph^3, \bar{g}_\infty) = \mathfrak{so}(4)$ isomorphic to $\mathfrak{so}(3)$. The convergence \eqref{eq:Xtildei-convergence}, combined with standard Gr\"onwall-type arguments, implies that the maps $\{\tilde{\Psi}_\rho\}_{\rho > R}$ are uniformly Cauchy in $C^{80}$. Hence, there is a smooth effective action $\tilde{\Psi}_\infty: \tilde{G}^0 \times \Sph^3 \to \Sph^3$ such that
    \begin{align} \label{eq:010968}
        \tilde{\Psi}_\rho \to \tilde{\Psi}_\infty \text{ in } C^{80}(\tilde{G}^0 \times \Sph^3).
    \end{align}
    Moreover, $\tilde{\Psi}_\infty$ acts by isometries with respect to $\bar{g}_\infty$. Thus $\tilde{\Psi}_\infty(\tilde{G}^0, \cdot)$ is a subgroup of $\I^0(\Sph^3, \bar{g}_\infty)$ isomorphic to $SO(3)$. By Lemma \ref{lem:subalg-classification-s3}, in particular the final line of Table \ref{table:2}, it follows that $\tilde{\Psi}_\infty(\tilde{G}^0, \cdot)$ acts by rotations of $\Sph^3 \subset \R^4$ fixing an axis. By composing $\tilde{\Psi}$ with a Euclidean rotation, we may assume the axis of rotation is the $w$-axis, i.e. it is spanned by $(0,0,0,1) \in \R^4$.
\end{proof}

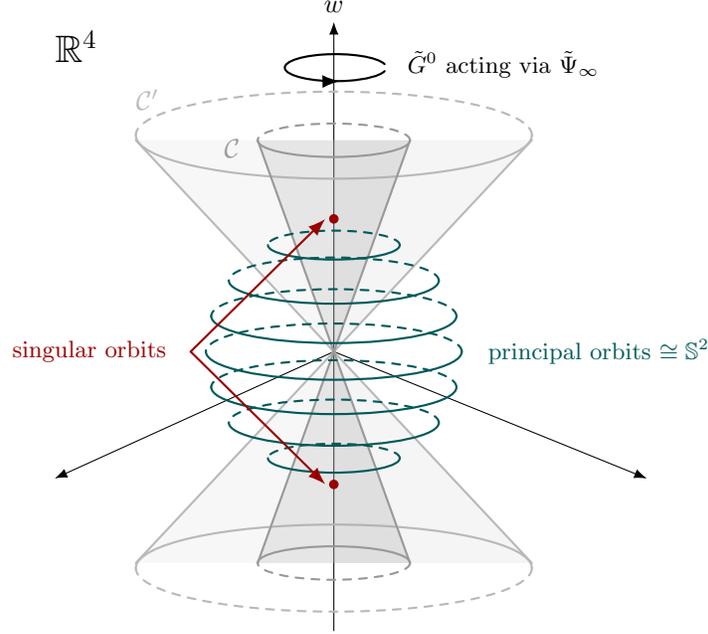
\begin{figure}[t]
    \centering
    \begin{tikzpicture}[scale=2.25, line cap=round, line join=round, >=Latex]

    \def\H{1.25}           
    \def\persp{0.22}       
    \def\xyshrink{0.72}    
    
    \pgfmathsetmacro{\kC}{0.5}
    \pgfmathsetmacro{\kCp}{1.3}
    
    \pgfmathsetmacro{\RC}{\xyshrink*\kC*\H}
    \pgfmathsetmacro{\RCp}{\xyshrink*\kCp*\H}
    \pgfmathsetmacro{\eC}{\persp*\RC}
    \pgfmathsetmacro{\eCp}{\persp*\RCp}
    
    \draw[->] (0,-1.65) -- (0,1.95) node[above] {$w$};
    
    \draw[->] (0,0) -- (1.85,-0.75);
    \draw[->] (0,0) -- (-1.65,-0.75);
    
    \def\Rrot{0.30}
    \def\Erot{0.08}
    \def\Wrot{1.70}

    \draw[
        postaction={decorate},
        decoration={markings, mark=at position 0.78 with {\arrow{Latex}}},
        thick
    ] (0,\Wrot) ++(\Rrot,0)
    arc[start angle=15,end angle=345, x radius=\Rrot, y radius=\Erot];

    \node[right] at (\Rrot+0.08,\Wrot) {\small $\tilde G^{0}$ acting via $\tilde{\Psi}_\infty$};

    \node[anchor=north west, font=\LARGE] at (-1.7,1.95) {$\mathbb{R}^4$};
    
    
    \path[fill=gray!35, opacity=0.22]
        (0,0) -- (\RCp,\H) -- (0,\H) -- cycle;
    \path[fill=gray!35, opacity=0.16]
        (0,0) -- (-\RCp,\H) -- (0,\H) -- cycle;
    
    \path[fill=gray!35, opacity=0.22]
        (0,0) -- (\RCp,-\H) -- (0,-\H) -- cycle;
    \path[fill=gray!35, opacity=0.16]
        (0,0) -- (-\RCp,-\H) -- (0,-\H) -- cycle;
    
    \draw[gray!55, thick]
        (0,\H+0.03) ++(\RCp,0) arc[start angle=0,end angle=-180, x radius=\RCp, y radius=\eCp];
    \draw[gray!55, thick, dashed]
        (0,\H+0.03) ++(\RCp,0) arc[start angle=0,end angle=180, x radius=\RCp, y radius=\eCp];
    \draw[gray!55, thick]
        (0,-\H-0.03) ++(\RCp,0) arc[start angle=0,end angle=180, x radius=\RCp, y radius=\eCp];
    \draw[gray!55, thick, dashed]
        (0,-\H-0.03) ++(\RCp,0) arc[start angle=0,end angle=-180, x radius=\RCp, y radius=\eCp];
    
    \draw[gray!55, thick]
        (0,0) -- (\RCp,\H) (0,0) -- (-\RCp,\H)
        (0,0) -- (\RCp,-\H) (0,0) -- (-\RCp,-\H);
    
    \node[gray!55] at (-1.10,\H+0.24) {$\mathcal{C}'$};
    
    
    \path[fill=gray!70, opacity=0.28]
        (0,0) -- (\RC,\H) -- (0,\H) -- cycle;
    \path[fill=gray!70, opacity=0.20]
        (0,0) -- (-\RC,\H) -- (0,\H) -- cycle;
    
    \path[fill=gray!70, opacity=0.28]
        (0,0) -- (\RC,-\H) -- (0,-\H) -- cycle;
    \path[fill=gray!70, opacity=0.20]
        (0,0) -- (-\RC,-\H) -- (0,-\H) -- cycle;
    
    \draw[gray!80, thick]
        (0,\H) ++(\RC,0) arc[start angle=0,end angle=-180, x radius=\RC, y radius=\eC];
    \draw[gray!80, thick, dashed]
        (0,\H) ++(\RC,0) arc[start angle=0,end angle=180, x radius=\RC, y radius=\eC];
    \draw[gray!80, thick]
        (0,-\H) ++(\RC,0) arc[start angle=0,end angle=180, x radius=\RC, y radius=\eC];
    \draw[gray!80, thick, dashed]
        (0,-\H) ++(\RC,0) arc[start angle=0,end angle=-180, x radius=\RC, y radius=\eC];
    
    \draw[gray!80, thick]
        (0,0) -- (\RC,\H) (0,0) -- (-\RC,\H)
        (0,0) -- (\RC,-\H) (0,0) -- (-\RC,-\H);
    
    \node[gray!80] at (-0.60,\H-0.05) {$\mathcal{C}$};
    

    \def\RSthree{1.05} 
    \def\wcompress{0.70} 
    \colorlet{orbitcol}{teal!70!black}
    \colorlet{singorbitcol}{red!60!black}

    \fill[singorbitcol] (0,\wcompress*\RSthree+0.05) circle (0.8pt);
    \fill[singorbitcol] (0,-\wcompress*\RSthree-0.05) circle (0.8pt);


    \foreach \wi in {-0.90,-0.60,-0.30,0.00,0.30,0.60,0.90}{
    \pgfmathsetmacro{\wdraw}{\wcompress*\wi}
    \pgfmathsetmacro{\rho}{\xyshrink*sqrt(\RSthree*\RSthree - \wi*\wi)}
    \pgfmathsetmacro{\erho}{\persp*\rho}

    \draw[orbitcol, thick] (0,\wdraw) ++(\rho,0)
        arc[start angle=0,end angle=-180, x radius=\rho, y radius=\erho];
    \draw[orbitcol, thick, dashed] (0,\wdraw) ++(\rho,0)
        arc[start angle=0,end angle=180, x radius=\rho, y radius=\erho];
    }

    \pgfmathsetmacro{\rhoZero}{\xyshrink*\RSthree}
    \node[orbitcol, right] at (\rhoZero+0.10,0)
        {\small principal orbits $\cong \Sph^2$};

    \node[singorbitcol, align=left] (singlab) at (-1.4,0)
    {\small singular orbits \;};

    \draw[singorbitcol,->, thick] (singlab.east) -- (-0.05,\wcompress*\RSthree+0.05);
    \draw[singorbitcol,->, thick] (singlab.east) -- (-0.05,-\wcompress*\RSthree-0.05);

    \end{tikzpicture}
    
    \caption{Limiting action $\tilde{\Psi}_\infty$ of $\tilde{G}^0 \cong SO(3)$ on $\R^4$, viewed with $w$ vertical and $x,y,z$ in the horizontal plane. The limiting action is by rotations fixing the $w$-axis; the actual action $\tilde{\Psi}$ of $\tilde{G}^0$ on $\tilde{M}_\infty$ is asymptotic to this. Examples of orbits, as well as the cones $\mathcal{C} = \{x^2 + y^2 + z^2 \leq w^2\}$ and $\mathcal{C}' = \{x^2 + y^2 + z^2 \leq 3w^2\}$ used in the proof of Corollary \ref{cor:drho-not-0}, are shown.}
    \label{fig:asymptotic-action}
\end{figure}

\begin{corollary} \label{cor:drho-not-0}
    Fix $\tilde{\Phi}: \tilde{M}_\infty \to \R^4 \setminus B_R(0)$ as in Lemma \ref{lem:limiting-action-Gtilde0}, and let
    \begin{align}
        \mathcal{C} = \{(x,y,z,w) \in \R^4: x^2 + y^2 + z^2 \leq w^2\}, \quad \Omega = (\R^4 \setminus B_R(0)) \setminus \mathcal{C}.
    \end{align}
    Then after increasing $R$ sufficiently (i.e. restricting the diffeomorphism $\tilde{\Phi}$), we have:
    \begin{enumerate}[label=(\alph*)]
        \item $\tilde{\Phi}^{-1}(\Omega) \subset \tilde{M}_\infty^{\mathrm{prin}}$.
        \item Let $\tilde{r}$ be the Riemannian distance from a point with respect to $\tilde{g}$. On $\tilde{\Phi}^{-1}(\Omega)$, we have $A^{-1}\tilde{r} \leq \xi \leq A\tilde{r}$ for some $A>0$, and $0 \neq |d\xi| = 1 + \O(\tilde{r}^{-4})$.
        \item Viewing $\tilde{\Phi}^{-1}(\Omega) \subset B \times \Sph^2$ via Proposition \ref{prop:prin-warped}, it contains an open set of the form $\Omega_B \times \Sph^2$, where $\Omega_B \subset B$ is an open neighborhood of a forward flow line of $\nabla\xi$.
    \end{enumerate}
\end{corollary}
\begin{proof}
    Let $\tilde{\Phi}$, $\tilde{\Psi}$, and $\tilde{\Psi}_\infty$ be as in Lemma \ref{lem:limiting-action-Gtilde0}. Moreover, extend $\tilde{\Psi}_\infty$ radially to act on $\R^4$; this is depicted in Figure \ref{fig:asymptotic-action}. By Lemma \ref{lem:limiting-action-Gtilde0}, the action $\tilde{\Psi}$ is asymptotic to $\tilde{\Psi}_\infty$.

    Note that all singular (i.e. non-principal) orbits of $\tilde{\Psi}_\infty$ lie on the $w$-axis. Therefore, outside a sufficiently large ball, all singular orbits of $\tilde{\Psi}$ are contained in a uniform tubular neighborhood of the $w$-axis, hence in $\mathcal{C}$. Thus, after increasing $R$ sufficiently, $\Omega$ contains only principal orbits. This establishes part (a) of the corollary.
    
    For each $p=(x,y,z,w) \in \R^4 \setminus B_R(0)$, the orbit of $p$ under $\tilde{\Psi}_\infty$ is a round $\Sph^2$ of radius $\xi_\infty(p) = \sqrt{x^2 + y^2 + z^2}$. In particular, along any outward Euclidean ray not parallel to the $w$-axis, the orbit radius $\xi_\infty$ increases at a definite linear rate. Therefore, there exists $A>0$ such that $A^{-1}|p| \leq \xi_\infty(p) \leq A|p|$ for all $p \in \Omega$. Thus, the same is true for the orbit radius function $\xi$ of $\tilde{\Psi}$ outside a sufficiently large ball. This proves the first claim of part (b). As for the second claim, we calculate the vertical and horizontal components of $\Ric = 0$ on $(\tilde{M}_\infty^{\mathrm{prin}} \cong B \times \Sph^2, \tilde{g} = h + \xi^2 g_{\Sph^2})$ to be
    \begin{align}
        \Ric_h &= 2\xi^{-1} (\nabla^h)^2\xi, \label{eq:RicH} \\
        \xi \Delta^h \xi + |d\xi|^2 &= 1. \label{eq:RicV}
    \end{align}
    (See e.g. \cite{besse}*{Proposition 9.106}.)
    Tracing \eqref{eq:RicH} with $h$ gives $\Delta^h \xi = \frac{1}{2} R_h \xi$, where $R_h$ is the scalar curvature of $h$. Plugging this into \eqref{eq:RicV} yields
    \begin{align}
        |d\xi|^2 = 1 - \frac{1}{2} R_h \xi^2.
    \end{align}
    Note that $|R_h| \leq C|{\Rm}_{\tilde{g}}|_{\tilde{g}} \leq \O(\tilde{r}^{-6})$ by Theorem \ref{thm:ALE-decay-rate}. Therefore, using the fact that $A^{-1}\tilde{r} \leq \xi \leq A\tilde{r}$ on $\tilde{\Phi}^{-1}(\Omega)$, we obtain
    \begin{align}
        |d\xi|^2 = 1 + \O(\tilde{r}^{-4}),
    \end{align}
    which proves the second claim of (b).

    Finally, let $\Omega' = (\R^4 \setminus B_R(0)) \setminus \mathcal{C}' \subset \Omega$, where $\mathcal{C}'$ is the widened cone
    \begin{align}
        \mathcal{C}' = \{(x,y,z,w) \in \R^4 : x^2 + y^2 + z^2 \leq 3w^2 \}.
    \end{align}
    Observe that $\Omega'$ contains $\Sph^2$ orbits of $\tilde{\Psi}$ which can be chosen arbitrarily far from the origin.
    Take one such orbit, and flow it forwards by the vector field $\nabla\xi$, which behaves like the outward Euclidean radial vector field near infinity. If the orbit was chosen sufficiently far from the origin, then its image under the flow never exits $\Omega$.
    Part (c) follows from these observations.
\end{proof}

Corollary \ref{cor:drho-not-0} allows us to use $\xi$ as a coordinate on $\Omega_B \times \Sph^2$. Since $\xi$ is constant on $\Sph^2$, it descends to a coordinate on $\Omega_B$. We now complete this to a full coordinate system using well-known arguments from the study of spherically symmetric spacetimes.
\begin{proposition} \label{prop:birkhoff}
    There is a domain $\Omega_B' \subset B$ which contains a forward flow line of $\nabla\xi$, and admits a coordinate system $(\xi, t): \Omega_B' \to \R^2$ mapping onto an open neighborhood $D \subset \R^2$ of $(R,\infty) \times \{0\}$ (for some large $R$). In this coordinate system, the metric $h$ can be written
    \begin{align} \label{eq:h-xi-t}
        h = E(\xi) d\xi^2 + F(\xi) dt^2
    \end{align}
    for some smooth positive functions $E$ and $F$.
\end{proposition}
\begin{proof}
    By Proposition \ref{prop:prin-warped}, Ricci-flatness, and a local version of Birkhoff's theorem (e.g. \cite{an-wong}*{Proposition 2.1}),
    $K := (\star_h d\xi)^\sharp$ is a Killing field on $(B,h)$ satisfying $\inner{K}{\nabla\xi} = 0$. Since $d\xi$ is nonvanishing on the set $\Omega_B \subset B$ from Corollary \ref{cor:drho-not-0}, so is $K$.

    By Corollary \ref{cor:drho-not-0}(c), $\Omega_B$ contains a forward flow line $\gamma$ of $\nabla\xi$. Around each point of $\gamma$, we can use the flow of $K$ to construct a local coordinate $t$ such that $K = \del_t$ and $\nabla\xi(t) = 0$. If two such coordinates $t_\alpha, t_\beta$ are defined on overlapping patches, then $u = t_\alpha - t_\beta$ satisfies $Ku = \nabla\xi(u) = 0$ on the overlap, which we may assume to be connected. Therefore, $u$ is a constant, so by adding a constant to $t_\beta$, it can be made to agree with $t_\alpha$ on the overlap.
    Therefore, there is a coordinate system $(\xi,t)$ for a neighborhood $\Omega_B'$ of $\gamma$ which maps to a neighborhood of $(R,\infty) \times \{0\}$ in $\R^2$ (for some large $R$). Since $\inner{\del_\xi}{\del_t} = \inner{\frac{\nabla\xi}{|\nabla\xi|^2}}{K} = 0$, we can write
    \begin{align}
        h = E(\xi,t) d\xi^2 + F(\xi,t) dt^2
    \end{align}
    for some positive functions $E$ and $F$. Since $K$ is a Killing field, $E$ and $F$ are independent of $t$.
\end{proof}

\begin{proposition} \label{prop:g-flat}
    There is a nonempty open set $U \subset \tilde{M}_\infty^{\mathrm{prin}}$ on which $\tilde{g}$ is flat.
\end{proposition}
\begin{proof}
    Propositions \ref{prop:prin-warped} and \ref{prop:birkhoff} imply that
    \begin{align} \label{eq:tilde-g-form}
        \tilde{g} = h + \xi^2 g_{\Sph^2} = E(\xi) d\xi^2 + F(\xi) dt^2 + \xi^2 g_{\Sph^2} \quad \text{on } D \times \Sph^2,
    \end{align}
    where $D(\xi,t) \subset \R^2$ is an open neighborhood of $(R,\infty) \times \{0\}$. Rudimentary computations yield
    \begin{align} \label{eq:hessh-xi}
        ((\nabla^h)^2 \xi)_{\xi\xi} = -\frac{E'}{2E}, \quad ((\nabla^h)^2 \xi)_{tt} = \frac{F'}{2E}.
    \end{align}
    Meanwhile, since $B$ is 2-dimensional, \eqref{eq:RicH} gives
    \begin{align}
        (\nabla^h)^2 \xi = \frac{1}{2} \xi \Ric_h = \frac{1}{4} \xi R_h h.
    \end{align}
    Combining this with \eqref{eq:hessh-xi} and then isolating $\frac{1}{2} \xi R_h$ in both equations, we obtain
    \begin{align}
        \frac{F'}{EF} = \frac{1}{2} \xi R_h = -\frac{E'}{E^2}.
    \end{align}
    This implies $(\log(EF))' = 0$, so $E = \frac{c}{F}$ for some $c > 0$. Thus, additionally defining $V(\xi) = F(\xi)/c$ and $\tau = \sqrt{c}t$, \eqref{eq:tilde-g-form} becomes
    \begin{align} \label{eq:tilde-g-form-2}
        \tilde{g} = \frac{1}{V(\xi)} d\xi^2 + V(\xi) d\tau^2 + \xi^2 g_{\Sph^2} \quad \text{on } D' \times \Sph^2,
    \end{align}
    where $D'(\xi,\tau) \subset \R^2$ is an open neighborhood of $(R,\infty) \times \{0\}$. Next, using \cite{besse}*{Proposition 9.106}, along the spherical direction we have
    \begin{align} \label{eq:Ricg2}
        {\Ric_{\tilde{g}}}|_{T\Sph^2} &= (1 - \xi \Delta^h \xi - |\nabla\xi|^2) g_{\Sph^2}.
    \end{align}
    Straightforward computation using \eqref{eq:tilde-g-form-2} gives $|\nabla\xi|^2 = V$ and $\Delta^h\xi = V'$. Inserting this into \eqref{eq:Ricg2}, imposing ${\Ric_{\tilde{g}}}|_{T\Sph^2} = 0$, and solving the resulting ODE gives
    \begin{align}
        V(\xi) = 1 - \frac{2\mu}{\xi} \quad \text{for some } \mu \in \R.
    \end{align}
    Thus, $\tilde{g}$ is the Euclidean Schwarzschild metric with mass $\mu$ on $D' \times \Sph^2$. Suppose $\mu \neq 0$. If we fix some $\theta \in \Sph^2$ and consider the curve $\xi \mapsto (\xi,0,\theta) \in D' \times \Sph^2$, then $|{\Rm}_{\tilde{g}}| \asymp \xi^{-3}$ as $\xi \to \infty$ (a standard fact about Euclidean Schwarzschild metrics with nonzero mass). By Corollary \ref{cor:drho-not-0}, $\xi$ is comparable to the Riemannian distance $\tilde{r}$ from a point with respect to $\tilde{g}$. Thus, $|{\Rm}_{\tilde{g}}| \asymp \tilde{r}^{-3}$ along the same curve. However, recall from the discussion after Corollary \ref{cor:Gdim} that $\tilde{g}$ is asymptotically Euclidean of order 4. So $|{\Rm}_{\tilde{g}}| \leq \O(\tilde{r}^{-6})$, but this is a contradiction. Therefore $\mu=0$, which makes $\tilde{g}$ flat in $D' \times \Sph^2$.
\end{proof}

We may now finish off Case (B). By Proposition \ref{prop:g-flat}, the downstairs metric $g$ is flat in an open set of $M$. Since Ricci-flat metrics are real-analytic, unique continuation implies that $(M,g)$ is flat everywhere. Hence, $(M,g)$ must be the flat orbifold $\R^4/\Z_2$. This concludes Case (B) and hence the proof of Theorem \ref{thm:EH-characterization}.

\section{A characterization of Calabi space} \label{sec:calabi-result}

In this section, we prove Theorem \ref{thm:calabi-characterization}.

\subsection{Homogeneous metrics on $\Sph^{2m-1}/\Z_m$ and isometry groups}

We begin by importing several facts from \cite{law}*{\S 5} regarding homogeneous metrics on $\Sph^{2m-1}/\Z_m$ and their isometry groups. Recall from \S\ref{subsec:calabi} that $\Sph^{2m-1}/\Z_m$ is the lens space obtained by quotienting $\Sph^{2m-1}$ by a $\Z_m$ subgroup of Hopf rotations.

\begin{lemma}[\cite{law}*{Lemma 5.3}] \label{lem:law-5.3}
    If $\mathfrak{g}$ is a Lie subalgebra of $\mathfrak{u}(m)$ satisfying
    \begin{align}
        \dim\mathfrak{g} \geq m^2 - 2m + 3,
    \end{align}
    then either $\mathfrak{g} = \mathfrak{su}(m)$ or $\mathfrak{g} = \mathfrak{u}(m)$. If $m=4$ we also allow $\mathfrak{g} = \mathfrak{sp}(2) \oplus \mathfrak{u}(1)$ (up to conjugacy in $\mathfrak{u}(4)$).
\end{lemma}

For the next result, recall the definition of $d_1$ stated before Lemma \ref{lem:exact-KF-subalgebra}. As before, $\bar{g}_\infty$ denotes the metric of constant sectional curvature 1.
\begin{proposition}[\cite{law}*{Proposition 5.4}] \label{prop:law-5.4}
    We have $d_1(\Sph^{2m-1}/\Z_m, \bar{g}_\infty) \geq 2m-3$.
\end{proposition}

Next, we recall the quotient-Berger metrics $\{g_{a,b}\}_{a,b>0}$ on $\Sph^{2m-1}/\Z_m$ defined in \S\ref{subsec:calabi}. The following result characterizes such metrics on $\Sph^{2m-1}/\Z_m$ in terms of isometry group dimensions.
\begin{proposition}[\cite{law}*{Proposition 5.5}] \label{prop:law-5.5}
    If $g$ is a homogeneous metric on $\Sph^{2m-1}/\Z_m$ with
    \begin{align}
        \dim \I(\Sph^{2m-1}/\Z_m, g) \geq \begin{cases}
            m^2 - 2m + 3 &\text{if } m=3 \text{ or } m \geq 5, \\
            14 &\text{if } m=4,
        \end{cases}
    \end{align}
    then $g = g_{a,b}$ for some $a,b > 0$. Thus $\I^0(\Sph^{2m-1}/\Z_m, g) = U(m)/\Z_m$, and the isotropy subgroup of its action is the image of $U(m-1) \times \Z_m$ under the quotient map $U(m) \to U(m)/\Z_m$.
\end{proposition}

\begin{lemma}[\cite{law}*{obtained by combining Lemma 5.7 and Corollary 5.8}] \label{lem:groups-SU}
    If $K$ is a closed subgroup of $SU(m)$ containing $S(U(m-1) \times \Z_m)$, and $K/S(U(m-1) \times \Z_m)$ is diffeomorphic to a spherical space form of dimension $\ell \geq 0$, then $K$ is one of $S(U(m-1) \times \Z_{2m})$, $S(U(m-1) \times U(1))$, or $SU(m)$.
\end{lemma}

\begin{lemma}[\cite{law}*{Lemma 5.9}] \label{lem:SUm-invariant-metrics}
    Any $SU(m)$-invariant metric on $SU(m)/S(U(m-1) \times \Z_m) \cong \Sph^{2m-1}/\Z_m$ is equal to $g_{a,b}$ for some $a,b > 0$.
\end{lemma}

\subsection{Symmetry reduction and the flat case of Theorem \ref{thm:calabi-characterization}}

The next result was also proved in \cite{law}*{\S 5} (specifically Corollary 5.6 there), but since it also depends on previous sections, we will write out the proof to avoid excessive cross-referencing.
\begin{corollary} \label{cor:SUm-action}
    Let $(M^{2m}, g, J)$ be a complete Ricci-flat K\"ahler ALE space of complex dimension $m \geq 3$ with group at infinity $\Z_m$ and satisfying \eqref{eq:calabi-thm-bound}. Then the group $SU(m)$ acts smoothly, almost effectively, and isometrically with cohomogeneity one on $(M,g)$. The principal isotropy subgroup is $S(U(m-1) \times \Z_m)$.
\end{corollary}
\begin{proof}
    We have $\I^0(\Sph^{2m-1}/\Z_m, \bar{g}_\infty) = U(m)/\Z_m$, whose dimension is $\bar{d} = m^2$. Thus, by Proposition \ref{prop:exact-KFs} and the remarks after it, there is a Lie subalgebra $\mathfrak{g}$ of Killing fields on $(M,g)$ with
    \begin{align} \label{eq:10100}
        \dim\mathfrak{g} \geq \bar{d} - \dim\ker_{L^2}(\Delta_L) \geq \begin{cases}
            m^2 - 2m + 3 &\text{if } m=3 \text{ or } m \geq 5, \\
            14 &\text{if } m=4,
        \end{cases}
    \end{align}
    which for large $\rho$ restricts to a Lie algebra $\mathfrak{g}_\rho$ of Killing fields on the CMC leaves $(\Sph^{2m-1}/\Z_m, g_\rho)$. These Lie algebras integrate to connected Lie groups $G \subseteq \I^0(M,g)$ and $G_\rho \subseteq \I^0(\Sph^{2m-1}/\Z_m, g_\rho)$ with dimensions also lower-bounded as in \eqref{eq:10100}.

    Proposition \ref{prop:law-5.4} implies $\dim\ker_{L^2}(\Delta_L) \leq d_1(\Sph^{2m-1}/\Z_m, \bar{g}_\infty)$. Hence, Lemma \ref{lem:exact-KF-subalgebra} implies that $\mathfrak{g}_\rho$ spans the tangent space to $\Sph^{2m-1}/\Z_m$ at any point if $\rho$ is sufficiently large. Hence, $(\Sph^{2m-1}/\Z_m, g_\rho)$ is homogeneous with $\dim\I(\Sph^{2m-1}/\Z_m, g_\rho)$ also lower bounded by \eqref{eq:10100}.
    
    By Proposition \ref{prop:law-5.5}, we have $g_\rho = g_{a,b}$ for some $a,b > 0$. Moreover, the transitive action of $\I^0(\Sph^{2m-1}/\Z_m, g_\rho) = U(m)/\Z_m$ has isotropy subgroup given by the image of $U(m-1) \times \Z_m$ under the quotient map $U(m) \to U(m)/\Z_m$. Since $\mathfrak{g}_\rho \subseteq \mathfrak{I}(\Sph^{2m-1}/\Z_m, g_\rho) = \mathfrak{u}(m)$ and the lower bound \eqref{eq:10100} holds for $\dim\mathfrak{g}_\rho$, Lemma \ref{lem:law-5.3} implies that $\mathfrak{g}_\rho \in \{\mathfrak{su}(m), \mathfrak{u}(m)\}$ (the $\mathfrak{sp}(2) \oplus \mathfrak{u}(1)$ case when $m=4$ is excluded by dimension). Hence, the integrated subgroup $G_\rho \subseteq \I^0(\Sph^{2m-1}/\Z_m, g_\rho)$ is the image of either $SU(m)$ or $U(m)$ under the quotient map $U(m) \to U(m)/\Z_m$. We consider these cases separately.

    In the former case, $G_\rho$ still acts transitively on $(\Sph^{2m-1}/\Z_m, g_\rho)$ with isotropy subgroup $H$ given by the image of $S(U(m-1) \times \Z_m)$ under the same quotient map. Then Lemma \ref{lem:level-set-isos} shows that $G$ acts smoothly, isometrically, and effectively on $(M,g)$, and the restriction map $G \to G_\rho$ is an isomorphism. Transitivity of $G_\rho$ on $(\Sph^{2m-1}/\Z_m, g_\rho)$ implies that $G$ acts by cohomogeneity one on $(M,g)$ with principal isotropy subgroup $H$. By undoing the $\Z_m$ quotients in the groups $G$ and $H$, the action of $G$ on $(M,g)$ becomes almost effective, and the corollary follows in this case.

    In the latter case, $G_\rho$ acts transitively on $(\Sph^{2m-1}/\Z_m, g_\rho)$ with isotropy subgroup $H$ given by the image of $U(m-1) \times \Z_m$ under the quotient map $U(m) \to U(m)/\Z_m$. As in the last paragraph, $G \cong G_\rho$ acts by cohomogeneity one on $(M,g)$ with principal isotropy subgroup $H$. Undoing the $\Z_m$ quotients, the corollary holds with $SU(m)$ replaced by $U(m)$ and $S(U(m-1) \times \Z_m)$ replaced by $U(m-1) \times \Z_m$. Nonetheless, by restricting this $U(m)$-action to $SU(m)$, the resulting action remains transitive on the CMC leaves. Therefore, the corollary still holds as stated.
\end{proof}

We now sharpen Corollary \ref{cor:SUm-action}; the following analysis is similar to that of \S\ref{subsec:caseA}.
\begin{proposition} \label{prop:bihol-metric}
    Let $(M^{2m}, g, J)$ be a complete Ricci-flat K\"ahler ALE space of complex dimension $m \geq 3$ with group at infinity $\Z_m$, and with orbifold singular set of real codimension at least 4. Suppose also that \eqref{eq:calabi-thm-bound} holds. Then exactly one of (A) or (B) below is true:
    \begin{enumerate}[label=(\Alph*)]
        \item $M$ is diffeomorphic to the total space of the complex line bundle $\O_{\CP^{m-1}}(-m)$. Moreover, identifying the complement of the zero section with $(0,\infty) \times \Sph^{2m-1}/\Z_m$, we have $g = ds^2 + g_{a(s),b(s)}$ where $a$ and $b$ are positive analytic functions with boundary conditions $a(0) = 0$, $a'(0) = m$, $b(0) > 0$, and $b'(0) = 0$.
        \item $(M,g,J)$ is biholomorphically isometric to the flat K\"ahler orbifold $\C^m/\Z_m$.
    \end{enumerate}
\end{proposition}
\begin{proof}
    Let $G = SU(m)$ and $H = S(U(m-1) \times \Z_m)$.
    By Corollary \ref{cor:SUm-action} and Proposition \ref{prop:coho-1-structure}, we get the following.
    \begin{enumerate}[label=(\roman*)]
        \item Outside a submanifold $X$, we can identify
        \begin{align}
            (M \setminus X, g) = \left( (0,\infty) \times G/H, ds^2 + g(s) \right),
        \end{align}
        where $\{g(s)\}_{s>0}$ are $G$-invariant metrics on a fixed copy of $G/H \cong \Sph^{2m-1}/\Z_m$.
        \item There is a closed subgroup $K$ of $G$ strictly containing $H$ such that $X$ is diffeomorphic to $G/K$, and $K/H$ is diffeomorphic to a spherical space form of dimension $\ell \geq 0$.
        \item From the description of $(M,g)$ as a bundle over $X$, the metric spaces $(G/H, g(s))$ converge in the Hausdorff sense to $X$ as $s \to 0$.
    \end{enumerate}
    From (i) and Lemma \ref{lem:SUm-invariant-metrics}, we have $g(s) = g_{a(s),b(s)}$ for some positive functions $a,b: (0,\infty) \to (0,\infty)$. By Ricci-flatness, $a$ and $b$ must be analytic. By (ii) and Lemma \ref{lem:groups-SU}, there are three possibilities for $K$ and $X \cong G/K$:
    \begin{enumerate}[label=(\alph*)]
        \item $K = S(U(m-1) \times \Z_{2m})$ and $X \cong \Sph^{2m-1}/\Z_{2m}$.
        \item $K = S(U(m-1) \times U(1))$ and $X \cong \CP^{m-1}$.
        \item $K = SU(m)$ and $X$ is a point.
    \end{enumerate}
    However, the first option is incompatible with (iii) because it is impossible for $(G/H, g_{a(s),b(s)})$ to converge in the Hausdorff sense to a metric space homeomorphic to $\Sph^{2m-1}/\Z_{2m}$. Indeed, any change in topology must arise from sending $a(s)$ or $b(s)$ to zero, which necessarily results in a dimension drop. Hence we are left with two possibilities:
    \begin{itemize}
        \item If $X \cong \CP^{m-1}$, then $M$ is a smooth manifold. Indeed, either all points in the orbit $X$ are orbifold points, or they are all smooth points. Since $\mathrm{codim}_M(X) = 2$, and $M$ has orbifold singular set of codimension at least 4, the latter must be true. As the principal part $(0,\infty) \times G/H$ is already smooth, it follows that $M$ is smooth. By construction, $M$ is diffeomorphic to $\O_{\CP^{m-1}}(-m)$.
        
        By (i) and the above discussion, the metric outside $X$ is $g = ds^2 + g_{a(s),b(s)}$. For the metric to degenerate to an $X \cong \CP^{m-1}$ at $s=0$, we need $a(s) \to 0$ and $b(s) \to b(0) > 0$ as $s \to 0$. For the metric to remain smooth at $X$, we require that $a'(0) = m$ and $b'(0) = 0$. This gives part (a) of the proposition.

        \item If $X$ is a point, then $M$ is diffeomorphic to the orbifold $\C^m/\Z_m$. Its universal orbifold cover is a complete Ricci-flat K\"ahler asymptotically Euclidean manifold $(\C^m, \tilde{g}, \tilde{J})$. By the sharp case of the Bishop--Gromov theorem, $\tilde{g}$ is the Euclidean metric. K\"ahlerity then necessitates that $(\C^m,\tilde{g},\tilde{J})$ is biholomorphically isometric to the standard $\C^m$.
    \end{itemize}
\end{proof}

Case (B) of Proposition \ref{prop:bihol-metric} establishes the flat $\C^m/\Z_m$ case of Theorem \ref{thm:calabi-characterization}. It remains to characterize $(M,g,J)$ as the Calabi space in case (A). So far, the metric is of the correct form, but more needs to be said about the complex structure. We do this in the next subsection and then finish the proof of Theorem \ref{thm:calabi-characterization}.

\subsection{The non-flat case of Theorem \ref{thm:calabi-characterization}}

For the rest of this section, assume that $(M^{2m},g,J)$ is as in case (A) of Proposition \ref{prop:bihol-metric}. Thus, outside the central $\CP^{m-1}$ we have
\begin{align}
    g = ds^2 + g_{a(s),b(s)} = ds^2 + a(s)^2 \sigma \otimes \sigma + b(s)^2 \pi^*g_{\mathrm{FS}},
\end{align}
where $a$ and $b$ are positive analytic functions on $(0,\infty)$. Write $M^{\mathrm{prin}} = M \setminus \CP^{m-1}$. The projection $\pi: \Sph^{2m-1}/\Z_m \to \CP^{m-1}$ induces a Riemannian submersion
\begin{align} \label{eq:submersion-pi}
    \tilde{\pi}: (M^{\mathrm{prin}}, g) \to ((0,\infty) \times \CP^{m-1}, ds^2 + b(s)^2 g_{\mathrm{FS}}).
\end{align}
Let $\xi$ be the vector field algebraically dual to $\sigma$, so that $\sigma(\xi) = 1$. Then $U = \frac{1}{a(s)}\xi$ is a unit vector field tangential to the Hopf fibers.
Fixing a point $x \in M^{\mathrm{prin}}$, we may decompose
\begin{align}
    T_x M = \mathcal{P}_x \oplus \mathcal{P}_x^\perp,
\end{align}
where $\mathcal{P}_x = \tilde{\pi}^*(T_{\tilde{\pi}(x)}\CP^{m-1})$ is the horizontal lift of $T_{\tilde{\pi}(x)}\CP^{m-1}$, and $\mathcal{P}_x^\perp = \mathrm{span}\{\del_s, U\}$ is its orthogonal complement. Note that the horizontal subspace at $x$ of the Riemannian submersion $\tilde{\pi}$ is $\mathcal{P}_x \oplus \mathrm{span}\{\del_s\}$, and the vertical subspace at $x$ is $\mathrm{span}\{U\}$.

Below, we denote by $I = \tilde{\pi}^*J_{\mathrm{FS}}$ the horizontal lift of the complex structure on $\CP^{m-1}$. Note that $I$ vanishes on $\del_s$, so it restricts to a complex structure on $\mathcal{P}_x$. We will prove the following characterization of $J$:
\begin{lemma} \label{lem:cpx-struc-formula}
    At all points $x$ in a dense open subset of $M^{\mathrm{prin}}$, we have $J = \epsilon_{\mathcal{P}} I + \epsilon_{\mathcal{P}^\perp} J_{\perp}$ for some $\epsilon_{\mathcal{P}}, \epsilon_{\mathcal{P}^\perp} \in \{1,-1\}$, where $J_{\perp}$ is defined by
    \begin{align}
        J_\perp(\del_s) = U, \quad J_\perp(U) = -\del_s, \quad J_\perp(\mathcal{P}_x) = 0.
    \end{align}
\end{lemma}
This will be proved by commuting curvature operators with $J$. Below, for $X,Y \in T_x M$, the operator $X \wedge Y \in \mathfrak{so}(T_x M)$ is defined by $(X \wedge Y)(Z) = \inner{Y}{Z} X - \inner{X}{Z} Y$.

\begin{lemma} \label{lem:curv-computation}
    Let $x \in M^{\mathrm{prin}}$. For each $X,Y,Z \in \mathcal{P}_x$, we have
    \begin{align}
        \Rm(X,Y)\del_s &= \left( \frac{2ab'}{b^3} - \frac{2a'}{b^2} \right) \inner{IX}{Y} U, \label{eq:curv-lemma1} \\
        \Rm(X,Y)U &= -\left( \frac{2ab'}{b^3} - \frac{2a'}{b^2} \right) \inner{IX}{Y} \del_s, \label{eq:curv-lemma2} \\
        \Rm(X,Y)Z &= \frac{(b')^2 - 1}{b^2} (X \wedge Y)(Z) + \left( \frac{a^2}{b^4} - \frac{1}{b^2} \right) (IX \wedge IY)(Z) \label{eq:curv-lemma3} \\
        &\quad + \left( \frac{2}{b^2} - \frac{2a^2}{b^4} \right) \inner{IX}{Y} IZ, \\
        \Rm(\del_s, U)\del_s &= -\frac{a''}{a} U, \label{eq:curv-lemma4} \\
        \Rm(\del_s, U)U &= \frac{a''}{a} \del_s, \label{eq:curv-lemma5} \\
        \Rm(\del_s, U)X &= \left( \frac{2ab'}{b^3} - \frac{2a'}{b^2} \right) IX. \label{eq:curv-lemma6}
    \end{align}
    In particular, $\Rm(\del_s,U)$ and $\Rm(X,Y)$ both preserve the splitting $T_x M = \mathcal{P}_x \oplus \mathcal{P}_x^\perp$.
\end{lemma}
\begin{proof}
    Let $X,Y,Z,W \in \mathcal{P}_x$. Applying O'Neill's formulae \cite{oneill} to the Riemannian submersion $\tilde{\pi}$, we have
    \begin{align}
        \inner{\Rm(X,Y)\del_s}{U} &= \inner{(\nabla_{\del_s}A)_X Y}{U} + \inner{A_X Y}{T_U \del_s} - \inner{A_Y \del_s}{T_U X} - \inner{A_{\del_s} X}{T_U Y}, \\
        \inner{\Rm(X,Y)Z}{U} &= \inner{(\nabla_Z A)_X Y}{U} + \inner{A_X Y}{T_U Z} - \inner{A_Y Z}{T_U X} - \inner{A_Z X}{T_U Y}, \\
        \inner{\Rm(X,Y)Z}{\del_s} &= \inner{\Rm^{\tilde{\pi}}(X,Y)Z}{\del_s} - 2 \inner{A_X Y}{A_Z \del_s} + \inner{A_Y Z}{A_X \del_s} + \inner{A_Z X}{A_Y \del_s}, \\
        \inner{\Rm(X,Y)Z}{W} &= \inner{\Rm^{\tilde{\pi}}(X,Y)Z}{W} - 2 \inner{A_X Y}{A_Z W} + \inner{A_Y Z}{A_X W} + \inner{A_Z X}{A_Y W},
    \end{align}
    where $\Rm^{\tilde{\pi}}(X,Y)Z = \tilde{\pi}^*(\Rm_0(\tilde{\pi}_*X, \tilde{\pi}_*Y) \tilde{\pi}_*Z)$ and $\Rm_0$ is the curvature of the base manifold of $\tilde{\pi}$. The tensors $A$ and $T$ are defined as follows: for $E,F \in \mathcal{P}_x$, we arbitrarily extend them to $\mathcal{P}$-valued vector fields around $x$ denoted by the same letters, and define
    \begin{align}
        A_E F = \frac{1}{2} \inner{[E,F]}{U} U, \quad T_U E = \inner{\nabla_U E}{U} U.
    \end{align}
    Meanwhile, we let
    \begin{align}
        A_E \del_s = \frac{1}{2} \inner{[E,\del_s]}{U} U, \quad T_U \del_s = \inner{\nabla_U \del_s}{U} U.
    \end{align}
    We will now compute all the $A$ and $T$ combinations appearing above. Below, we extend $X \in \mathcal{P}_x$ to a local $\mathcal{P}$-valued vector field by first extending in the $\del_s$ direction with $[\del_s,X] = 0$, then horizontally lifting in the $\xi$ direction so that $[\xi,X] = 0$. It follows that
    \begin{align}
        [U,X] = \frac{1}{a(s)}[\xi,X] - X\left(\frac{1}{a(s)}\right)\xi = 0.
    \end{align}
    We do the same for $Y,Z \in \mathcal{P}_x$.
    
    Firstly, since $[X,\del_s] = 0$, we have
    \begin{align} \label{eq:curv1}
        A_X \del_s = 0.
    \end{align}
    Next, since $U = \frac{1}{a}\xi$ where $\xi$ is algebraically dual to $\sigma$, and $\sigma(X) = \sigma(Y) = 0$, we have $A_X Y = \frac{a}{2} \sigma([X,Y]) U = -\frac{a}{2} d\sigma(X,Y) U$. Combined with the fact that $d\sigma = 2 \pi^*\omega_{\mathrm{FS}} = 2\pi^*g_{\mathrm{FS}}(J_{\mathrm{FS}}(\cdot), \cdot)$ for the Hopf fibration $\pi: \Sph^{2m-1}/\Z_m \to \CP^{m-1}$, this implies
    \begin{align} \label{eq:curv2}
        A_X Y = -\frac{a}{b^2} \inner{IX}{Y} U
    \end{align}
    where we recall that $I := \tilde{\pi}^*J_{\mathrm{FS}}$.
    Since $(\CP^{m-1}, g_{\mathrm{FS}}, J_{\mathrm{FS}})$ is K\"ahler and $\tilde{\pi}$ is a Riemannian submersion, we have $\inner{(\nabla_Z I)X}{Y} = 0$. Differentiating \eqref{eq:curv2} gives
    \begin{align}
        \nabla_Z(A_X Y) = -\frac{a}{b^2} \left[ \inner{I(\nabla_Z X)}{Y}U + \inner{IX}{\nabla_Z Y} U + \inner{IX}{Y} \nabla_Z U \right],
    \end{align}
    which implies that
    \begin{align} \label{eq:curv3}
        (\nabla_Z A)_X Y &= \nabla_Z(A_X Y) - A_{\nabla_Z X}Y - A_X(\nabla_Z Y) = -\frac{a}{b^2} \inner{IX}{Y} \nabla_Z U.
    \end{align}
    Likewise, since $\nabla_{\del_s} X = \frac{b'}{b}X$ and $\nabla_{\del_s} Y = \frac{b'}{b}Y$, we compute
    \begin{align} \label{eq:curv4}
        (\nabla_{\del_s} A)_X Y &= \nabla_{\del_s} (A_X Y) = -\left(\frac{a}{b^2}\right)' \inner{IX}{Y} U - \frac{a}{b^2} \inner{IX}{Y} \nabla_{\del_s} U. 
    \end{align}
    Next, since $[U,X] = 0$ and $U$ has unit length, we have
    \begin{align} \label{eq:curv5}
        T_U X = \inner{\nabla_U X}{U} U = \inner{\nabla_X U}{U} U = \frac{1}{2} X|U|^2 = 0.
    \end{align}
    Meanwhile, since $[\xi,\del_s] = 0$ and $U = \frac{1}{a}\xi$, we have
    \begin{align} \label{eq:curv6}
        T_U \del_s = \frac{1}{a^2} \inner{\nabla_\xi \del_s}{\xi} U = \frac{1}{a^2}\inner{\nabla_{\del_s} \xi}{\xi} U = \frac{1}{2a^2} \del_s |\xi|^2 U = \frac{a'}{a} U.
    \end{align}

    We take all the computations between \eqref{eq:curv1} and \eqref{eq:curv6}, and substitute them back into O'Neill's formulas stated at the beginning of the proof. We also use $\inner{\nabla_{(\cdot)} U}{U} = 0$ repeatedly since $U$ is a unit vector field. This gives
    \begin{align}
        \inner{\Rm(X,Y)\del_s}{U} &= \left(-\left(\frac{a}{b^2}\right)' - \frac{a'}{b^2}\right) \inner{IX}{Y} = \left( \frac{2ab'}{b^3} - \frac{2a'}{b^2} \right) \inner{IX}{Y}, \\
        \inner{\Rm(X,Y)Z}{U} &= 0, \label{eq:curv-a2} \\
        \inner{\Rm(X,Y)Z}{\del_s} &= \inner{\Rm^{\tilde{\pi}}(X,Y)Z}{\del_s} = 0, \label{eq:curv-a3} \\
        \inner{\Rm(X,Y)Z}{W} &= \inner{\Rm^{\tilde{\pi}}(X,Y)Z}{W} - \frac{2a^2}{b^4} \inner{IX}{Y} \inner{IZ}{W} \\
        &\quad + \frac{a^2}{b^4} \inner{IY}{Z} \inner{IX}{W} + \frac{a^2}{b^4} \inner{IZ}{X} \inner{IY}{W}. \label{eq:curv-a4}
    \end{align}
    The first three equations, together with the symmetries of the curvature tensor, imply \eqref{eq:curv-lemma1}, \eqref{eq:curv-lemma2}, and \eqref{eq:curv-lemma6}. Meanwhile, the last three equations imply
    \begin{align} \label{eq:simplified-curv}
        \Rm(X,Y)Z &= \Rm^{\tilde{\pi}}(X,Y)Z - \frac{2a^2}{b^4} \inner{IX}{Y} IZ + \frac{a^2}{b^4} (IX \wedge IY)(Z).
    \end{align}
    Since the base manifold of \eqref{eq:submersion-pi} is itself a warped product, we use the formulae for the curvature of a warped product and the Fubini--Study metric (e.g. \cite{kn2}*{\S IX.7}) to get
    \begin{align}
        \Rm^{\tilde{\pi}}(X,Y)Z &= \tilde{\pi}^*(\Rm_0(\pi_*X,\pi_*Y)\pi_*Z) + \frac{(b')^2}{b^2} (X \wedge Y)(Z) \\
        &= \left(\frac{(b')^2-1}{b^2}\right)(X \wedge Y)(Z) + \frac{2}{b^2}\inner{IX}{Y} IZ - \frac{1}{b^2}(IX \wedge IY)(Z). \label{eq:cpn-curv}
    \end{align}
    Combining \eqref{eq:simplified-curv} and \eqref{eq:cpn-curv}, we arrive at \eqref{eq:curv-lemma3}.

    Finally, the identities $[\xi, \del_s] = 0$ and $U = \frac{1}{a}\xi$ imply that $[U,\del_s] = -\frac{a'}{a}U$. Since $\nabla_{\del_s} \del_s = 0$ and $\nabla_{\del_s} X = \frac{b'}{b}X$ is orthogonal to $U$, we get $\nabla_{\del_s} U = 0$. Thus, $\nabla_U \del_s = -[\del_s, U] = \frac{a'}{a}U$. These identities yield (after simplification)
    \begin{align}
        \Rm(\del_s, U)\del_s &= \nabla_U \nabla_{\del_s} \del_s - \nabla_{\del_s} \nabla_U \del_s + \nabla_{[\del_s,U]} \del_s = -\frac{a''}{a} U,
    \end{align}
    which implies \eqref{eq:curv-lemma4} and \eqref{eq:curv-lemma5}.
\end{proof}

Note that $I$ turns $\mathcal{P}_x \cong T_x\CP^{m-1}$ into a complex vector space, and together with the metric $g = \inner{\cdot}{\cdot}$ it induces a Hermitian inner product $(X,Y) \mapsto \inner{X}{Y} + i\inner{X}{IY}$. Denote by $U(\mathcal{P}_x, I, g)$ the group of linear transformations on the complex vector space $(\mathcal{P}_x, I)$ which preserve this Hermitian inner product, and $SU(\mathcal{P}_x, I, g)$ the subgroup of those with determinant 1. Let $\mathfrak{u}(\mathcal{P}_x, I, g)$ and $\mathfrak{su}(\mathcal{P}_x, I, g)$ be their Lie algebras.

\begin{corollary} \label{cor:subalg-su}
    For all $x$ in a dense open subset of $M^{\mathrm{prin}}$, there are Lie subalgebras $\mathfrak{h} \subset \mathfrak{k}$ of $\mathfrak{so}(T_x M)$ such that
    \begin{enumerate}[label=(\roman*)]
        \item Each map in $\mathfrak{k}$ is a linear combination of maps of the form $\Rm(V,W)$ where $V, W \in T_x M$.
        \item Each map in $\mathfrak{k}$ preserves the decomposition $T_x M = \mathcal{P}_x \oplus \mathcal{P}_x^\perp$.
        \item $\mathfrak{k}$ acts on $\mathcal{P}_x$ as the standard representation of $\mathfrak{u}(\mathcal{P}_x, I, g)$.
        \item $\mathfrak{h}$ acts on $\mathcal{P}_x^\perp$ trivially, and acts on $\mathcal{P}_x$ as the standard representation of $\mathfrak{su}(\mathcal{P}_x, I, g)$.
    \end{enumerate}
\end{corollary}
\begin{proof}
    Let $\{e_j\}_{j=1}^{m-1}$ be a unitary orthonormal basis for $(\mathcal{P}_x, I)$, so that $\{e_j, Ie_j\}_{j=1}^{m-1}$ is a $g$-orthonormal basis. Using Lemma \ref{lem:curv-computation}, we compute for each $j \neq k$
    \begin{align}
        A_{jk} &:= \Rm(e_j,e_k) + \Rm(Ie_j, Ie_k) = \left( \frac{(b')^2}{b^2} - \frac{2}{b^2} + \frac{a^2}{b^4} \right) (e_j \wedge e_k + Ie_j \wedge Ie_k), \\
        B_{jk} &:= \Rm(e_j,Ie_k) - \Rm(Ie_j,e_k) = \left( \frac{(b')^2}{b^2} - \frac{2}{b^2} + \frac{a^2}{b^4} \right) (e_j \wedge Ie_k - Ie_j \wedge e_k), \\
        D_{jk} &:= \Rm(e_j,Ie_j) - \Rm(e_k,Ie_k) = \left( \frac{(b')^2}{b^2} - \frac{2}{b^2} + \frac{a^2}{b^4} \right) (e_j \wedge Ie_j - e_k \wedge Ie_k).
    \end{align}
    Assume that the following are true at $x$; we will justify this later.
    \begin{align} \label{eq:nonzero-assumption-x}
        \frac{(b')^2}{b^2} - \frac{2}{b^2} + \frac{a^2}{b^4} \neq 0, \quad \frac{2ab'}{b^3} - \frac{2a'}{b^2} \neq 0.
    \end{align}
    Then $A_{jk}$, $B_{jk}$, and $D_{jk}$ vanish on $\mathcal{P}_x^\perp$, while they act on $\mathcal{P}_x$ as nonzero scalar multiples of
    \begin{align}
        A_{jk}&: \quad e_j \mapsto -e_k, \quad e_k \mapsto e_j, \quad Ie_j \mapsto -Ie_k, \quad Ie_k \mapsto Ie_j, \\
        B_{jk}&: \quad e_j \mapsto -Ie_k, \quad e_k \mapsto -Ie_j, \quad Ie_j \mapsto e_k, \quad Ie_k \mapsto e_j, \\
        D_{jk}&: \quad e_j \mapsto -Ie_j, \quad e_k \mapsto Ie_k, \quad Ie_j \mapsto e_j, \quad Ie_k \mapsto -e_k.
    \end{align}
    Thus, under the identification of complex vector spaces $(\mathcal{P}_x, I) \to (\C^{m-1}, i)$ defined by the $e_i$'s, these operators correspond to a basis for $\mathfrak{su}(m-1)$ as follows:
    \begin{align}
        A_{jk} \mapsto E_{jk} - E_{kj}, \quad B_{jk} \mapsto -i(E_{jk} + E_{kj}), \quad D_{jk} \mapsto i(E_{kk} - E_{jj}),
    \end{align}
    where $E_{jk}$ is the $(m-1) \times (m-1)$ matrix with $1$ in the $jk$-th entry and $0$ elsewhere. We define
    \begin{align}
        \mathfrak{h} := \mathrm{span}\{A_{jk}, B_{jk}, D_{jk}\}_{j,k=1,\ldots,m-1; \, j \neq k} \subset \mathfrak{so}(T_x M).
    \end{align}
    By the above discussion, $\mathfrak{h}$ satisfies item (iv) of the corollary.

    Next, let $C := \Rm(\del_s, U)$ and define
    \begin{align}
        \mathfrak{k} := \mathfrak{h} \oplus \mathrm{span}\{C\} \subset \mathfrak{so}(T_x M).
    \end{align}
    By \eqref{eq:nonzero-assumption-x} and Lemma \ref{lem:curv-computation}, $C$ acts as a nonzero scalar multiple of $I$ on $\mathcal{P}_x$. In view of item (iv) of the corollary, this implies item (iii) of the corollary. By construction, $\mathfrak{k}$ also satisfies item (i) as well as (by Lemma \ref{lem:curv-computation}) item (ii). Thus, the corollary is proved at points $x$ where \eqref{eq:nonzero-assumption-x} holds.
    
    Therefore, it remains to show that \eqref{eq:nonzero-assumption-x} on a dense open subset of $M$. This is equivalent to asking that $\varphi := b^2(b')^2 - 2b^2 + a^2 \neq 0$ and $\psi := 2ab' - 2a'b \neq 0$ on a dense open subset of $(0,\infty)$. Note that $\varphi$ is analytic on $(0,\infty)$ by Proposition \ref{prop:bihol-metric}, and the boundary conditions on $a$ and $b$ give $\lim_{s \to 0} \varphi(s) = -2b(0)^2 < 0$. In particular, $\varphi$ is nonzero for sufficiently small $s$, so it is nonzero on a dense open set. The same reasoning shows that $\psi$ is nonzero on a dense open set. Thus, $\varphi$ and $\psi$ are both nonzero on a dense open set.
\end{proof}

\begin{proof}[Proof of Lemma \ref{lem:cpx-struc-formula}]
    Since $(M,g,J)$ is K\"ahler, we have $\nabla J = 0$. Thus, $J$ commutes with all curvature operators $\Rm(X,Y): T_x M \to T_x M$. In particular, let $x \in M^{\mathrm{prin}}$ and $\mathfrak{h} \subset \mathfrak{k} \subset \mathfrak{so}(T_x M)$ be as in Corollary \ref{cor:subalg-su}. Then $J$ commutes with both $\mathfrak{h}$ and $\mathfrak{k}$.
    
    We have $J(\mathcal{P}_x^\perp) \subseteq \mathcal{P}_x^\perp$. Indeed, suppose to the contrary that for some $V \in \mathcal{P}_x^\perp$ we have $JV = V' + W$, where $V' \in \mathcal{P}_x^\perp$ and $W \in \mathcal{P}_x \setminus \{0\}$. By Corollary \ref{cor:subalg-su}, we can find $T \in \mathfrak{h}$ such that $T(W) \neq 0$. Since $\mathfrak{h}$ is trivial on $\mathcal{P}_x^\perp$, we have $T(V) = T(V') = 0$. Then
    \begin{align}
        0 = J(0) = JT(V) = T(JV) = T(V'+W) = T(W) \neq 0,
    \end{align}
    a contradiction. Hence $J(\mathcal{P}_x^\perp) \subseteq \mathcal{P}_x^\perp$. Since $J$ is invertible and orthogonal, this implies $J(\mathcal{P}_x^\perp) = \mathcal{P}_x^\perp$ and $J(\mathcal{P}_x) = \mathcal{P}_x$, i.e. $J$ preserves the splitting $T_x M = \mathcal{P}_x \oplus \mathcal{P}_x^\perp$.
    
    The commutant of $\mathfrak{u}(m-1)$ in $\mathrm{End}_\R(\C^{m-1})$ consists precisely of complex scalar multiples of the identity. As $\mathfrak{k}$ acts as the standard representation of $\mathfrak{u}(\mathcal{P}_x, I, g) \cong \mathfrak{u}(m-1)$ on $\mathcal{P}_x \cong \C^{m-1}$, its commutant in $\mathrm{End}_\R(\mathcal{P}_x)$ also consists precisely of complex scalar multiples of the identity, where $I$ is the complex structure. Therefore, $J|_{\mathcal{P}_x}$, which commutes with $\mathfrak{k}$, is of the form $(a+bI) \cdot \mathrm{Id}$ for some $a,b \in \R$; since it squares to $-\mathrm{Id}$, we must have $J|_{\mathcal{P}_x} = \pm I$.
    
    Meanwhile, since $\dim_\C \mathcal{P}_x^\perp = 1$ and $J|_{\mathcal{P}_x^\perp}$ is orthogonal, the only two possibilities are $J|_{\mathcal{P}_x^\perp} = J_\perp$ and $J|_{\mathcal{P}_x^\perp} = -J_\perp$, where $J_\perp$ is defined in the lemma.
\end{proof}

We can now finish proving Theorem \ref{thm:calabi-characterization}. The flat case was already handled in case (B) of Proposition \ref{prop:bihol-metric}, so it remains to analyze case (A). In this case, $(M,g,J)$ is diffeomorphic to $\O_{\CP^{m-1}}(-m)$, and
\begin{align} \label{eq:metric-form}
    g = ds^2 + g_{a(s),b(s)} = ds^2 + a(s)^2 \sigma \otimes \sigma + b(s)^2 \pi^*g_{\mathrm{FS}}
\end{align}
on $M^{\mathrm{prin}} = M \setminus \CP^{m-1}$, where $a$ and $b$ are positive analytic functions of $s \in (0,\infty)$.
Using Lemma \ref{lem:cpx-struc-formula} and \eqref{eq:metric-form}, we compute the K\"ahler form on a dense open subset of $M^{\mathrm{prin}}$ to be
\begin{align} \label{eq:kahler-form}
    \omega = g(J(\cdot), \cdot) = \epsilon_{\mathcal{P}^\perp} a(s) ds \wedge \sigma + \epsilon_{\mathcal{P}} b(s)^2 \pi^*\omega_{\mathrm{FS}},
\end{align}
where $\omega_{\mathrm{FS}}$ is the Fubini--Study K\"ahler form, and $\epsilon_{\mathcal{P}}, \epsilon_{\mathcal{P}^\perp}$ are $\{1,-1\}$-valued functions. Since \eqref{eq:kahler-form} holds in a dense open set of $M^{\mathrm{prin}}$, and $\omega$ is continuous, \eqref{eq:kahler-form} actually holds everywhere on $M^{\mathrm{prin}}$. Moreover, since $M^{\mathrm{prin}}$ is connected, the functions $\epsilon_{\mathcal{P}}$ and $\epsilon_{\mathcal{P}^\perp}$ are constant.
By imposing the K\"ahler condition $d\omega = 0$ and using that $d\sigma = 2\pi^*\omega_{\mathrm{FS}}$, we get
\begin{align}
    \left( -2\epsilon_{\mathcal{P}^\perp} a + 2\epsilon_{\mathcal{P}} bb' \right) ds \wedge \pi^*\omega_{\mathrm{FS}} = 0.
\end{align}
As $a,b > 0$, and $b'$ is positive outside a compact set due to the ALE condition, this implies $\epsilon_{\mathcal{P}^\perp} = \epsilon_{\mathcal{P}}$ outside a compact set and hence everywhere. Thus,
\begin{align}
    \omega = \pm \left( a(s) ds \wedge \sigma + b(s)^2 \pi^*\omega_{\mathrm{FS}} \right)
\end{align}
globally on $M^{\mathrm{prin}}$. Up to possible negation of the complex structure, which can be achieved by pulling back using an antiholomorphic isometry, we have arrived exactly at the Calabi ansatz; see \S\ref{subsec:calabi}, in particular \eqref{eq:calabi-kahler}. From the discussion there, it follows that $(M,g,J)$ is biholomorphically isometric to the Calabi space of dimension $m$. This completes the proof of Theorem \ref{thm:calabi-characterization}.

\subsection*{Acknowledgements}

The author wishes to thank his advisor, William Minicozzi, for his continual encouragement and support. He would also like to thank Dain Kim for feedback on an earlier draft of this paper. This work was carried out under the support of a Simons Dissertation Fellowship.

\bibliography{Refs}
\end{document}